\documentclass[12pt,reqno]{amsart}

\usepackage{amsfonts}
\usepackage{hyperref}
\usepackage{appendix}
\usepackage{graphicx}

\usepackage{color}
\usepackage{amsmath}
\usepackage{mathabx}
\usepackage{bm}

\renewcommand{\div}{\rm div}

\newtheorem*{theorem-non}{Theorem}
\newtheorem*{lemma-non}{Lemma}
\newtheorem{theorem}{Theorem}[section]
\newtheorem{lemma}[theorem]{Lemma}
\newtheorem{corollary}[theorem]{Corollary}

\numberwithin{equation}{section}

\theoremstyle{definition}
\newtheorem{definition}[theorem]{Definition}
\newtheorem{example}[theorem]{Examples}
\newtheorem{remark}[theorem]{Remark}
\newtheorem{theorem*}[theorem]{Theorem}

\def\O{\Omega}
\def\C{\mathcal{C}}

\def\div{{\rm div}\,}
\def\supp{{\rm supp}}
\def\dx{{\,\rm d}x}
\def\dy{{\,\rm d}y}

\def\v{{\bf v}}
\def\gve{\varepsilon}

\newcommand{\R}{{\mathbb R}}
\newcommand{\N}{{\mathbb N}}

\DeclareMathOperator*{\esssup}{ess\,sup}
\DeclareMathOperator*{\essinf}{ess\,inf}

\begin{document}
\date{\today}
\title[Discrete Hardy inequalities on trees and applications]{Weighted discrete Hardy inequalities on trees and applications}

\author[F. L\'opez-Garc\'\i a]{Fernando L\'opez-Garc\'\i a}
\address{Department of Mathematics and Statistics\\ California State Polytechnic University Pomona\\
3801 West Temple Avenue\\ Pomona, CA 91768, US} 
\email{fal@cpp.edu}

\author[I. Ojea]{Ignacio Ojea}
\address{Departamento de Matem\'atica, Facultad de Ciencias Exactas y Naturales, Universidad de Buenos Aires and IMAS, CONICET, 1428 Buenos Aires, Argentina} 
\email{iojea@dm.uba.ar}

\thanks{The first author is supported by Cal Poly Pomona under start-up funds. The second author is supported by ANCyPT under grant PICT 2018 - 3017, by CONICET under grant PIP112201130100184CO and by Universidad de Buenos Aires under grant 20020170100056BA}



\begin{abstract}

 In this paper, we study certain inequalities and a related result for weighted Sobolev spaces on H\"older-$\alpha$ domains, where the weights are powers of the distance to the boundary. We obtain results regarding the divergence equation's solvability, and the improved Poincar\'e, the fractional Poincar\'e, and the Korn inequalities. The proofs are based on a local-to-global argument that involves a kind of atomic decomposition of functions and the validity of a weighted discrete Hardy-type inequality on trees. The novelty of our approach lies in the use of this weighted discrete Hardy inequality and a sufficient condition that allows us to study the weights of our interest. As a consequence, the assumptions on the weight exponents that appear in our results are weaker than those in the literature.


\end{abstract}
\maketitle

\section{Introduction}
\label{Intro}

Let $\{U_t\}_{t\in\Gamma}$ be a certain partition of a bounded domain $\Omega\subset\R^n$. Given $f\in L^1(\Omega)$ with vanishing mean value, we decompose it into the sum of a collection of functions $\{f_t\}_{t\in\Gamma}$, where $f_t$ is supported on $U_t$ and has vanishing mean value. This kind of decomposition was applied by Bogovskii in \cite{Bogovskii} using a finite partition to extend the solvability of the divergence equation from star-shaped domains with respect to a ball to Lipschitz domains. In the articles \cite{DRS} and \cite{DMRT}, the authors used a similar decomposition where the partition of the domain is countable. In the case where the partition is not finite, it is required to have an upper bound of the sum of the norms of $\{f_t\}_{t\in\Gamma}$ by the norm of the function $f$. In \cite{DRS}, the decomposition is developed for John domains, and the estimation of the norms is based on the continuity of the Hardy-Littlewood maximal operator. In \cite{DMRT}, the authors considered more general domains and the decomposition is based on the validity of a certain Poincar\'e-type inequality. This decomposition can be used for extending to general domains several results that are known to hold on simpler ones, e.g.: the solvability of the divergence equation, and the inequalities Poincar\'e, improved Poincar\'e, fractional  Poincar\'e and Korn. The decomposition presented here is based on the one developed in \cite{L1} where a continuous Hardy-type inequality is applied for proving the estimation for the norms. Moreover, in \cite{L1} the partition of the domain is indexed over a set $\Gamma$ with tree structure, which is strongly related to the geometry of the domain. Other references where variations of these techniques are used are: \cite{D,L2,L3,HL,DL}. 

Also, in \cite{AO,BLV} a similar decomposition is used on cuspidal domains for proving weighted Korn inequalities. In those papers, thanks to the geometry of the domain, the partition is indexed over $\N$ (in other words, it is formed by a chain of subdomains). The discrete weighted Hardy-type inequality \cite[inequality (1.102), page 56]{KPS} is used for proving the estimate of the norms. 

In this work, we are interested in having a better understanding of the weights that make these inequalities valid. We apply a discrete approach, similar to the one used in \cite{AO,BLV}, i.e.: our partition of the domain allows us to regard the weights as essentially constant over each sub-domain and a discrete Hardy-type inequality is used for estimating a weighted norm of the sum of $\{f_t\}$ in terms of another weighted norm of $f$. On the other hand, we recover the tree structure introduced in \cite{L1}, which allows the method to be applied to a larger class of domains. Hence, we need a discrete weighted Hardy-type inequality, similar to \cite[inequality (1.102)]{KPS}, but for sequences indexed over trees. For this inequality to hold, necessary and sufficient conditions on the weights can be derived from the continuous case, treated in \cite{EHP}. However, as we shall discuss below, these conditions are very hard to check for our examples. Hence, we prove a sufficient condition which is somehow a \emph{natural} extension of the classical condition for sequences and is much easier to verify.

The paper is organized as follows: Section 2 introduces the weighted discrete Hardy-type inequality that is applied later, and provides two conditions on the weights that imply its validity. In Section 3, we present our decomposition of functions with vanishing mean value on arbitrary bounded domains. We also show how the Hardy-type inequality stated in the previous section can be used to obtain an upper bound of the norms of the functions proposed in the decomposition. In Section 4, we study the decomposition of functions defined in Section 3 on bounded H\"older domains. In Section 5, we prove several interesting results that are obtained as a consequence of the decomposition. In particular, we prove the solvability of the divergence equation and improved Poincar\'e, fractional Poincar\'e and Korn inequalities. All these results are stated on weighted Sobolev spaces on bounded H\"older domains, where the weights are powers of the distance to the boundary. In all cases, the conditions imposed on the exponents of the weights are less restrictive than the ones in the literature.
In Appendix A, we derive from \cite{EHP} a necessary and sufficient condition for the validity of the weighted discrete Hardy-type inequality treated in this work. This condition is included in the manuscript for general knowledge, but it is not used in our applications.

\section{A weighted discrete Hardy inequality on trees}
\label{Hardy on trees}

In this section, we study a certain weighted Hardy-type inequality on trees, and give two conditions for its validity. The first condition is sufficient and necessary and it follows from \cite{EHP} (see Theorem \ref{ehp theorem}). The second condition is sufficient, and it may also be necessary, but we haven't proven it. We are especially interested in this second one because its verification in our examples is easier than the first one. 

Throughout the paper $1< p,q<\infty$, with $\frac{1}{p}+\frac{1}{q}=1$, unless otherwise stated. 

A tree is a graph $(V,E)$, where $V$ is the set of vertices and $E$ the set of edges, satisfying that it is connected and has no cycles. A tree is said to be rooted if one vertex is designed as root. In a rooted tree $(V,E)$, it is possible to define a {\it partial order} ``$\preceq$" in $V$ as follows: $s\preceq t$ if and only if the unique path connecting $t$ to the root $a$ passes through $s$. {\it The parent} $t_p$ of a vertex $t$ is the vertex connected to $t$ by an edge on the path to the root. It can be seen that each $t\in V$ different from the root has a unique parent, but several elements ({\it children}) on $V$ could have the same parent. We assume that each vertex has a finite number of children. Note that two vertices are connected by an edge ({\it adjacent vertices}) if one is the parent of the other one. We say that a set of indices $\Gamma$ has a tree structure if there is a set of edges such that $(\Gamma,E)$ is a rooted tree. 

Trees can be regarded as continuous or as discrete. In a continuous tree, the edges are segments on the plane, and one can define functions taking values over them, whereas the set of vertices has vanishing measure. On the other hand, on discrete trees, the edges are just links between the vertices that define a partial order. In this case, sequences indexed on the vertices can be defined. There is a natural one to one map between the edges and the subset of vertices $\Gamma\setminus\{a\}$. It is given by the association of the edge $(t_p,t)$ with the vertex $t$. This map implies an association between the continuous and  discrete versions of a given tree. Therefore, we define:
\[\Gamma^* = \Gamma\setminus\{a\}.\]
We will work with discrete trees which are derived from a continuous setting, so $\Gamma^*$ is the natural environment for stating our Hardy-type inequality. It is important to notice, however, that the same results that we present here on $\Gamma^*$ can be easily extended to $\Gamma$.

Given a rooted tree $\Gamma$, we consider collections of real values indexed over $\Gamma^*$, named in this work as $\Gamma^*$-sequences.  We define $\ell^p(\Gamma^*)$ the space of collections ${\bm b} = \{b_t\}_{t\in\Gamma^*}$ such that:
\[\|{\bm b}\|_p= \left(\sum_{t\in\Gamma^*}b_t^p\right)^\frac{1}{p}<\infty.\]
We also define $\mathcal{P}_t$ the path from the root $a$ to $t$: \[\mathcal{P}_t:=\{s:a\prec s\preceq t\},\]
and $\mathcal{S}_t$ the \emph{shadow} of $t$: \[\mathcal{S}_t:=\{s\in\Gamma^*:\, s\succeq t\}.\]

Given positive $\Gamma^*$-sequences (i.e. weights) ${\bm u} = \{u_t\}_{t\in\Gamma^*}$ and ${\bm v} = \{v_t\}_{t\in\Gamma^*}$, we introduce the inequality: 
\begin{equation}\label{dHardy dual general}
\left(\sum_{t\in\Gamma^*}\bigg|u_t^{-1}\sum_{s\succeq t} b_s\bigg|^q\right)^\frac{1}{q}\le C\left(\sum_{t\in\Gamma^*}\left|b_t v_t^{-1}\right|^q\right)^{\frac{1}{q}},
\end{equation}
for every ${\bm b}=\{b_t\}_{t\in\Gamma}$ such that ${\bm b}{\bm v}^{-1}\in \ell^q(\Gamma^*).$

Notice that the following dual version to \eqref{dHardy dual general} is equivalent. 
\begin{lemma}
Inequality \eqref{dHardy dual general} holds if and only if
\begin{equation}\label{dHardy primal}
\left(\sum_{s\in\Gamma^*} \left|v_s\sum_{a\prec t\preceq s} d_t\right|^p\right)^\frac{1}{p} \le C \left(\sum_{s\in\Gamma^*} \left|d_s u_s\right|^p\right)^{\frac{1}{p}}
\end{equation}
 is satisfied, for every ${\bm d}=\{d_t\}_{t\in\Gamma}$ such that ${\bm d}{\bm u}\in \ell^p(\Gamma^*)$. Moreover, the optimal constants for both inequalities are equal to each other.
\end{lemma}

\begin{proof}
The best constant $C$ for \eqref{dHardy dual general} can be characterized by duality as: 
 \begin{align*}
C &= \sup_{{\bm \delta}:\|{\bm\delta}\|_p=1}\ \sup_{{\bm b}:\|{\bm b}{\bm v}^{-1}\|_q=1} \sum_{t\in \Gamma^*}u_t^{-1}\Big(\sum_{s\succeq t}b_s\Big)\delta_t \\
  &= \sup_{{\bm b}:\|{\bm b}{\bm v}^{-1}\|_q=1}\ \sup_{{\bm \delta}:\|{\bm\delta}\|_p=1} \sum_{s\in \Gamma^*}b_s\Big(\sum_{a\prec t\preceq s}\delta_t u_t^{-1}\Big).
\end{align*} 
Now, taking ${\bm d}= {\bm \delta}{\bm u}^{-1}$ and ${\bm \beta} = {\bm b} {\bm v}^{-1}$, we obtain the dual characterization of the optimal constant in \eqref{dHardy primal}:
 \begin{align*}
C&= \sup_{{\bm \beta}:\|{\bm \beta}\|_q=1}\ \sup_{{\bm d}:\|{\bm d}{\bm u}\|_p=1} \sum_{s\in \Gamma^*} v_s\Big(\sum_{a\prec t\preceq s}d_t\Big)\beta_s.
\end{align*}
\qed\end{proof}

The goal of this section is to establish conditions for \eqref{dHardy dual general} to hold that can be verified in our examples. 

It is known (see for example \cite{KMP}, \cite{KPS}, \cite{O}) that the classical necessary and sufficient condition for the continuous Hardy inequality in an interval translates to the discrete case. Namely, if $\Gamma$ is a \emph{chain} (i.e. a tree where each vertex has at most one child), then inequalities \eqref{dHardy primal} and \eqref{dHardy dual general} hold if and only if
\begin{equation}\label{chain cond}
A_{chain} = \sup_{t\in\Gamma^*} \left(\sum_{a\prec s\preceq t}u_s^{-q}\right)^\frac{1}{q}\left(\sum_{s\succeq t}v_s^p\right)^\frac{1}{p}<\infty.
\end{equation}
Moreover, the constant $C$ in \eqref{dHardy dual general} is proportional to $A_{chain}$. 

The authors in \cite{EHP} studied continuous Hardy inequalities on trees, where their main result can be easily translated to the discrete case as shown in the following theorem. However, they also showed that, on trees that are not chains, condition \eqref{chain cond} is necessary for the validity of \eqref{dHardy dual general}, but not sufficient. 

\begin{theorem}\label{ehp theorem}
Let $\Gamma$ be a discrete tree with root $a$. Given $K$ a subtree of $\Gamma$, we define its boundary as $\partial K = \{s\in K:\; \exists t,\, t_p = s,\, t\notin K\}.$ We also define the following class formed by some subtrees of $\Gamma$:
\[\mathcal{K}=\{K \textrm{ subtree of }\Gamma:\, a\in K,\, \textrm{ and if } s\in\partial K, \textrm{ then } t\notin K, \, \forall t\succ s\}.\]
For $K\in\mathcal{K}$, we define the interior of $K$, $K^\circ=K\setminus \partial K$. 
Then inequality \eqref{dHardy dual general} holds if and only if: 
\begin{equation}\label{EHP cond}
B := \sup_{K\in\mathcal{K}} \frac{\|\bm{v}\|_{\ell^p(\Gamma\setminus K^\circ)}}{\alpha_K}<\infty,
\end{equation}
where:
\[\alpha_K := \inf\Big\{\|{\bm b}\|_{p}:\; \sum_{a\prec s\preceq t}|b_s|u_s^{-1} = 1, \,\forall t\in \partial K\Big\}.\]
Moreover, the constant $C$ in \eqref{dHardy dual general} is proportional to $B$.
\end{theorem}

Condition \eqref{EHP cond} is rather cumbersome and one can find it very hard to prove in practical examples. However, valuable information can be derived from it. E.g., fixing a vertex $t\in\Gamma^*$, consider the sub-tree: $K=\Gamma\setminus \mathcal{S}_t$. In this case, $\|\bm{v}\|_{\ell^p(\Gamma^*\setminus K)}=\|\bm{v}\|_{\ell^p(\mathcal{S}_t)}$. On the other hand, $t$ is the only vertex in $\partial K$ and $\alpha_K^{-1}$ becomes a dual characterization of $\|\bm{u}^{-1}\|_{\ell^q(\mathcal{P}_t)}$. Hence, the expression inside the supremum of \eqref{EHP cond} becomes the expression inside the supremum of \eqref{chain cond}, which proves: $A_{chain}\le B$. The converse, however is not true: in \cite[Section 5]{EHP} an example is given where $B=\infty$ whereas $A_{chain}$ remains bounded. 

In \cite{EHP} a recursive method for computing $\alpha_K$ is given, as well as several sufficient and slightly less complex conditions. But the main difficulty, namely: the necessity of estimating a supremum over \emph{all} subtrees in $\mathcal{K}$, remains. Hence, we prove  in the following Theorem a sufficient condition that can be regarded as a generalization of \eqref{chain cond}, and which is almost as easy to check. On the downside, we were not able to compare our sufficient condition with the other sufficient conditions given in \cite{EHP}. 

\begin{theorem}\label{suff cond}
Let ${\bm u} = \{u_t\}_{t\in\Gamma^*}$ and ${\bm v} = \{v_t\}_{t\in\Gamma^*}$ be two weights that satisfy:
\begin{equation}\label{tree cond}
A_{tree}:=\sup_{t\in\Gamma^*}\left(\sum_{a\prec s\preceq t} u_s^{-q}\right)^\frac{1}{\theta q}\left(\sum_{s\succeq t}v_s^p \bigg(\sum_{a\prec r\preceq s}u_r^{-q}\bigg)^{\frac{p}{q}(1-\frac{1}{\theta})}\right)^\frac{1}{p}<\infty,
\end{equation}
for some $\theta>1$. Then inequality \eqref{dHardy primal} holds. In addition, the optimal constant in  \eqref{dHardy primal} satisfies that $C\leq \Big(\frac{\theta}{\theta-1}\Big)^\frac{1}{q}A_{tree}.$
\end{theorem}
\begin{proof}
We follow an idea used in \cite{M}. However, the introduction of the parameter $\theta$ is crucial for obtaining a sharper result. We can assume that $d_t\geq0$.

We begin observing that the concavity of the function $f(x) = x^{1-\frac{1}{\theta}}$ implies, via the mean value theorem, the following inequality for $0\le x_1< x_2:$
\begin{equation}\label{concavity}
\frac{x_2-x_1}{x_2^\frac{1}{\theta}} \le \frac{\theta}{\theta-1}(x_2^{1-\frac{1}{\theta}}-x_1^{1-\frac{1}{\theta}}).
\end{equation}
Now, let us define $N(t) := \sum_{a\prec r\preceq t}u_r^{-q}$. Applying H\"older inequality we obtain: 
 \begin{align*}
I &:= \sum_{s\in\Gamma^*}v_s^p\Big(\sum_{a\prec t\preceq s} d_t\Big)^p = \sum_{s\in\Gamma^*}v_s^p\Big(\sum_{a\prec t\preceq s} d_t u_t N(t)^{\frac{1}{\theta q}} u_t^{-1}N(t)^{-\frac{1}{\theta q}}\Big)^p \\
&\le\sum_{s\in\Gamma^*}v_s^p\Big(\sum_{a\prec t\preceq s} d_t^p u_t^p N(t)^{\frac{p}{\theta q}}\Big)\Big(\sum_{a\prec t\preceq s} u_t^{-q}N(t)^{-\frac{1}{\theta}}\Big)^\frac{p}{q}.
\end{align*} 
For the last factor, observe that $u_t^{-q} = N(t)-N(t_p)$ for every $t$, where  $N(a)$ is defined as $0$. This and  \eqref{concavity} give: 
\[u_t^{-q}N(t)^{-\frac{1}{\theta}}=\frac{N(t)-N(t_p)}{N(t)^{\frac{1}{\theta}}}\le \frac{\theta}{\theta-1}\big(N(t)^{1-\frac{1}{\theta}}-N(t_p)^{1-\frac{1}{\theta}}\big).\]

Now, we apply a telescopic argument along the path that goes from $a$ to $s$, obtaining: 
 \begin{align*}
I&\le \Big(\frac{\theta}{\theta-1}\Big)^\frac{p}{q}\sum_{s\in\Gamma^*}v_s^p\Big(\sum_{t\preceq s} d_t^p u_t^p N(t)^{\frac{p}{\theta q}}\Big)\Big(\sum_{t\preceq s}\big(N(t)^{1-\frac{1}{\theta}}-N(t_p)^{1-\frac{1}{\theta}}\big)\Big)^\frac{p}{q}\\
&= \Big(\frac{\theta}{\theta-1}\Big)^\frac{p}{q}\sum_{s\in\Gamma^*}v_s^p\Big(\sum_{a\prec t\preceq s} d_t^p u_t^p N(t)^{\frac{p}{\theta q}}\Big)N(s)^{\frac{p}{q}(1-\frac{1}{\theta})}.
\end{align*} 
Interchanging the summations and applying condition $\eqref{tree cond}$:

 \begin{align*}
I&\le \Big(\frac{\theta}{\theta-1}\Big)^\frac{p}{q}\sum_{t\in\Gamma^*}d_t^p u_t^p \Big[N(t)^\frac{p}{\theta q} \sum_{s\succeq t} v_s^p N(s)^{\frac{p}{q}(1-\frac{1}{\theta})}\Big] \\
&\le \Big(\frac{\theta}{\theta-1}\Big)^\frac{p}{q}A_{tree}^p\sum_{t\in\Gamma^*}d_t^p u_t^p,
\end{align*} 
and the result follows. 
\qed\end{proof}

\begin{remark}
Condition \eqref{tree cond}, with the parameter $\theta>1$, resembles similar sufficient conditions that appear when dealing with weighted inequalities involving two weights. See, for example \cite[Theorem 1]{SW}.
\end{remark}

\begin{remark}\label{rmk integrability}
Observe that condition \eqref{tree cond} (as well as \eqref{chain cond}) implies that $\sum_{s\in\Gamma^*}v_s^p<\infty$. In other words: ${\bm v}\in \ell^p(\Gamma^*)$.
\end{remark}

\begin{remark}
The proof for Theorem \ref{suff cond} can be copied verbatim replacing $\Gamma^*$ by $\Gamma$, both in the inequality \eqref{dHardy primal} and in the condition \eqref{tree cond}. 
\end{remark}

Observe that as $\theta$ approaches $1$, condition \eqref{tree cond} ``tends" to condition \eqref{chain cond}. It seems that we cannot take $\theta\to 1$, since the factor $\frac{\theta}{\theta-1}$ goes to infinity. However, condition \eqref{tree cond} is actually equivalent to \eqref{chain cond}, if $\Gamma$ is a chain. Indeed: 

\begin{theorem}\label{equivalence}
If $\Gamma$ is a chain, then conditions \eqref{chain cond} and \eqref{tree cond} are equivalent.
\end{theorem}
\begin{proof}
 \eqref{tree cond} implies the validity of \eqref{dHardy primal}, which is equivalent to \eqref{chain cond} on chains, proving that \eqref{tree cond} implies \eqref{chain cond}. 

Suppose now that \eqref{chain cond} holds. Then: 
\[\Big(\sum_{a\prec r\preceq s}u_r^{-q}\Big)^\frac{1}{q}\le A_{chain}\Big(\sum_{r\succeq s}v_r^p\Big)^{-\frac{1}{p}},\]
which in turn gives: 
 \begin{align*}
\sum_{s\succeq t} v_s^p\Big(\sum_{a\prec r\preceq s}u_r^{-q}\Big)^{\frac{p}{q}(1-\frac{1}{\theta})}\le A_{chain}^{p(1-\frac{1}{\theta})}\sum_{s\succeq t}v_s^p\Big(\sum_{r\succeq s}v_r^p\Big)^{\frac{1}{\theta}-1}.
\end{align*} 

Now, let us assume the following inequality holds on chains for any $\theta>1:$
\begin{equation}\label{cond 2}
\sum_{s\succeq t}v_s^p\Big(\sum_{r\succeq s}v_r^p\Big)^{\frac{1}{\theta}-1}\le \theta\Big(\sum_{s\succeq t}v_s^p\Big)^\frac{1}{\theta}.
\end{equation}
Using this, we obtain: 
 \begin{align*}
A_{tree} &= \sup_{t\in\Gamma^*} \Big(\sum_{a\prec s\preceq t}u_s^{-q}\Big)^{\frac{1}{\theta q}}\Big(\sum_{a\prec s\preceq t}v_s^p\big(\sum_{a\prec r\preceq s}u_r^{-q}\big)^{\frac{p}{q}(1-\frac{1}{\theta})}\Big)^\frac{1}{p}  \\
&\le \theta^\frac{1}{p} A_{chain}^{1-\frac{1}{\theta}}\sup_{t\in\Gamma^*}\Big(\sum_{a\prec s\preceq t} u_s^{-q}\Big)^\frac{1}{\theta q} \Big(\sum_{s\succeq t}v_s^p\Big)^\frac{1}{p\theta}\le \theta^\frac{1}{p} A_{chain}.
\end{align*} 
Hence, it only remains to prove \eqref{cond 2}. Let us first present the main idea of why \eqref{cond 2} holds naturally on every chain. Suppose that we are working on a continuous setting. In that case, the left member of \eqref{cond 2} would become:
\[I = \int_t^\infty v(s)^p \Big(\int_s^\infty v(x)^p \dx\Big)^{\frac{1}{\theta}-1}\textrm{d}s.\]
Now, through the substitution $\xi = F(s) = \int_s^\infty v(x)^p \dx$, d$\xi = -v(s)^p\textrm{d}s$, we have:
\[I = -\int_{F(t)}^0 \xi^{\frac{1}{\theta}-1}\textrm{d}\xi = \theta \xi^\frac{1}{\theta}|_0^{F(t)} = \theta F(t)^\frac{1}{\theta} = \theta\Big(\int_t^\infty v(s)^p\textrm{d}s\Big)^\frac{1}{\theta},\]
which is the continuous analog to the right hand side of \eqref{cond 2}.

Now, in the discrete case, we cannot change variables as we did with the integral, but an adapted version of the same idea can be applied. We proceed in a similar way than the proof of Theorem \ref{suff cond}: we define $M(s) = \sum_{r\succeq s} v_r^p$. Recalling Remark \ref{rmk integrability}, we have that $M(s)<\infty$ and $\lim_{s\to\infty} M(s)=0$. We denote $s_c$ the child of $s$ along $\Gamma$, which is unique thanks to the fact that $\Gamma$ is a chain. Applying the convexity of the function $f(x) = x^\frac{1}{\theta}$ and a telescopic argument, we obtain: 
 \begin{align*}
\sum_{s\succeq t} v_s^p \Big(\sum_{r\succeq s}v_r^p\Big)^{\frac{1}{\theta}-1} &= 
\sum_{s\succeq t} \frac{M(s)-M(s_c)}{M(s)^{1-\frac{1}{\theta}}}\le \theta\sum_{s\succeq t}M(s)^\frac{1}{\theta}-M(s_c)^\frac{1}{\theta} \\
&= \theta M(t)^\frac{1}{\theta} = \theta\Big(\sum_{s\succeq t} v_s^p\Big)^\frac{1}{\theta},
\end{align*} 
which  completes the proof. 

Observe that the fact that each $s$ has only one child $s_c$ is crucial for the telescopic argument to hold. On general trees, this step cannot be performed, and the proof fails. 
\qed\end{proof}  

The proof of the previous Theorem shows that condition \eqref{tree cond} can be unravel into two: the classical condition \eqref{chain cond}, and the additional \eqref{cond 2}. We state this as a corollary, although for the application considered here it is easier to work directly with \eqref{tree cond}.

\begin{corollary}
Let $\Gamma$ be a tree. If ${\bm u}=\{u_t\}_{t\in\Gamma}$ and ${\bm v}=\{v_t\}_{t\in\Gamma}$ verify both \eqref{chain cond} and \eqref{cond 2} (with any constant), then \eqref{tree cond} holds. 
\end{corollary}

\section{A decomposition of functions}
\label{Decomposition}
\setcounter{equation}{0}

Let $\Omega\subset\R^n$ be a bounded domain with $n\geq 2$. We refer by a {\it weight} $\eta:\Omega\to\R$ to a Lebesgue-measurable function, which is positive almost everywhere.Then, we define the weighted spaces $L^p(\Omega,\eta)$ as the space of Lebesgue-measurable functions $f:\Omega\to\R$ with finite norm 
\[\|f\|_{L^p(\Omega,\eta)}=\left(\int_\Omega |f(x)|^p \eta(x)\, {\rm d}x\right)^{1/p}.\]
 Henceforth, $d,\,d_A:\Omega\to \R$ will denote the distance functions to $\partial\Omega$ and $A\subset \overline{\Omega}$ respectively.

\begin{definition}\label{Definition decomposition} Let $\mathcal{C}$ be the space of constant functions 
from $\R^n$ to $\R$
and  $\{U_t\}_{t\in\Gamma}$ a collection of open subsets of $\Omega$ that covers $\Omega$ except for a set of Lebesgue measure zero; $\Gamma$ is an index set.  It also satisfies the additional requirement that for each $t\in\Gamma$ the set $U_t$ intersects a finite number of $U_s$ with $s\in\Gamma$. This collection $\{U_t\}_{t\in\Gamma}$ is called an {\it open covering of $\Omega$}. Given $g\in L^1(\Omega)$ orthogonal to $\C$ (i.e., $\int g\, \varphi=0$ for all $\varphi\in\C$), we say that a collection of functions $\{g_t\}_{t\in\Gamma}$
  in $L^1(\Omega)$ is a $\C$-{\it orthogonal decomposition of $g$} subordinate to $\{U_t\}_{t\in\Gamma}$ if the following three properties are satisfied:
\begin{enumerate}
\item $g=\sum_{t\in \Gamma} g_t.$
\item $\supp (g_t)\subset U_t.$
\item $\int_{U_t} g_t=0$, for all $t\in\Gamma$.
\end{enumerate}
\end{definition}
We also refer to this collection of functions by a $\C$-{\it decomposition}.  
 Notice that condition (3) is equivalent to the orthogonality to the space $\C$ of constant functions. Indeed, this condition can be replaced by $\int_{U_t} g_t(x)\varphi(x)\dx=0$, for all $\varphi\in\C$ and $t\in\Gamma$. 
 
In Theorem \ref{Decomp Thm} below, we show the existence of a $\C$-orthogonal decomposition by using a constructive argument introduced in \cite{L1}. 

 \begin{definition}\label{admissible weights} Given a countable open covering $\{U_t\}_{t\in\Gamma}$ of $\Omega$, we say that a weight $\eta:\Omega\to \R$ is {\it admissible} if there exists a uniform constant $C$ such that 
\begin{equation}\label{admissible}
\esssup_{x\in U_t} \eta(x)\leq C \essinf_{x\in U_t} \eta(x),
\end{equation}
for all $t \in \Gamma$. Notice that admissible weights are subordinate to $\{U_t\}_{t \in \Gamma}$ of $\Omega$ and $1<p<\infty$. 
 \end{definition}

\begin{example}\label{Whitney} One classical example is induced by a Whitney decomposition. Given $\Omega\subset\R^n$ an open set, it is known (see, for example \cite[Section VI]{S_SingularIntegrals}), that there exists a collection of dyadic closed cubes, $\{Q_j\}_{j\in\N}$, with edges parallel to the coordinate axis, such that  $\Omega=\bigcup_{j} Q_j$, satisfying that the length $\ell(Q_j)$ of the cube $Q_j$ is proportional to $d(Q_j,\partial\Omega)$, where the constants involved does not depend on $j$. Moreover two neighbouring cubes are of similar size. 
These properties are well adapted for working with weights that depend on the distance to the boundary. Then, every weight $\eta(x)=d(x)^\tau$, with $\tau$ in $\R$, is admissible subordinate to $\{U_j\}_j$, where $U_j = \frac{17}{16}Q_j^\circ$. A construction similar to a Whitney decomposition is used in \cite{L1}, and in Section \ref{Holder}.  
\end{example}

\begin{example}\label{cusp} Another example is the one studied in the articles \cite{AO,BLV,L1}, where $\Omega$ is a cuspidal domain with only one singularity (the tip of the cusp) on its boundary. For example, we can consider
\[\Omega:=\{(x_1,x_2)\in\R^2\,\colon\,0<x_1<1\text{ and }0<x_2<x_1^\gamma\},\]
where $\gamma>1$.
In this case, it is of interest to consider weights that depend on the distance to the cusp instead of the distance to the boundary. For that reason, the partition of the domain depends on the singularity we have at the origin as it can be seen at the open covering $\{U_n\}_{n\geq 0}$:
\[U_n=\{(x_1,x_2)\in\Omega\,\colon\,2^{-(n+2)}<x_1<2^{-n}\}.\]
For this open covering, any power $\eta(x)=d_0(x)^\tau$ of the distance to the cusp is admissible.
\end{example}

 \begin{definition} \label{Decomp of Omega}
Let $\Omega\subset\R^n$ be a bounded domain. We say that an open covering $\{U_t\}_{t\in\Gamma}$ is
a {\it tree covering} of $\Omega$ if it also satisfies the properties: 
\begin{enumerate}
\item $\chi_\Omega(x)\leq \sum_{t\in\Gamma}\chi_{U_t}(x)\leq N \chi_\Omega(x)$, for almost every $x\in\Omega$, where $N\geq 1$.
\item   The set of subindices $\Gamma$ has the structure of a rooted tree, i.e. it is the set of vertices of a rooted tree $(\Gamma,E)$ with a
 root $a$.
\item There is a collection $\{B_t\}_{t\neq a}$ of pairwise disjoint open sets with $B_t\subseteq U_t\cap U_{t_p}$.
\end{enumerate}
 \end{definition}

\begin{remark}
Given an open covering of a domain $\Omega$, one can choose an element of the covering as the root, and there are different ways to define a tree-covering. Notice that two vertices on the tree are adjacent only if the intersection of their corresponding open sets is non-empty. Some care should be taken in order to obtain a meaningful tree-covering, according with the geometry of the domain. For example, it is known that the quasi-hyperbolic distance between two cubes in a Whitney decomposition is comparable with the shorter chain of cubes connecting them. Hence, on an open covering like the one in Example \ref{Whitney} we can define a tree-covering by an inductive argument on the quasi-hyperbolic distance to the root: this is done in \cite{H}. Another possible tree-covering on a Whitney decomposition can be defined when the domain is a John domain, in which case each chain connecting a Whitney cube with the root is a Boman chain. This type of tree-covering, which characterizes John domains, is introduced in \cite{L2}.

The open covering for external cusps in Example \ref{cusp} can be seen as a tree-covering that is actually a chain, with the root defined as the open set furthest from the tip of the cusp. 
\end{remark}

Given a tree covering $\{U_t\}_{t\in\Gamma}$ of $\O$ and $\nu,\omega:\O\to\R$ admissible weights subordinate to $\{U_t\}_{t\in \Gamma}$, we define the following discrete Hardy-type inequality on trees for positive sequences $\{b_t\}_{t\in\Gamma}$ 
\begin{equation}\label{dHardy dual}
\left(\sum_{t\in\Gamma^*} |B_t|^{-q/p} \nu_t^{-q}\left(\sum_{s\succeq t} b_s\right)^q \right)^{\frac{1}{q}} \leq C \left(\sum_{t\in\Gamma^*} |B_t|^{-q/p} \omega_t^{-q} b_t^q\right)^{\frac{1}{q}},
\end{equation}
 where the sequence weights $\{\omega_t\}_{t\in\Gamma^*}$ and  $\{\nu_t\}_{t\in\Gamma^*}$ are defined as 
\begin{equation*}
    \omega_t=\essinf_{x\in B_t} \omega(x)\quad \text{and} \quad \nu_t=\essinf_{x\in B_t} \nu(x).
\end{equation*}
Observe that here is where the necessity of working on $\Gamma^*$ becomes apparent, since the weights depend on $B_t$, which plays the role of the edge between $t_p$ and $t$, and is not defined for the root of the tree.

\begin{remark}\label{rmk weights}
Observe that \eqref{dHardy dual} is exactly \eqref{dHardy dual general}, taking $u_t=|B_t|^{\frac{1}{p}}\nu_t$ and $v_t = |B_t|^{\frac{1}{p}}\omega_t$.
\end{remark}

\begin{theorem}\label{Decomp Thm}  Let $\Omega\subset\R^n$ be a bounded domain with a tree covering  $\{U_t\}_{t\in \Gamma}$ such that $\frac{|U_t|}{|B_t|}\le C$ for every $t\in\Gamma^*$, and let $\nu,\omega:\O\to\R$ be admissible weights, with $\omega^p\in L^1(\Omega)$, such that  $L^q(\O,\omega^{-q})\hookrightarrow L^q(\O,\nu^{-q})$ and the weighted discrete Hardy inequality on trees \eqref{dHardy dual} holds. Then, given $g$ in $L^q(\Omega,\omega^{-q})$, with $\int_\O g=0$, there exists $\{g_t\}_{t\in\Gamma}$, a  $\C$-decomposition of $g$, such that 
\begin{equation}\label{Decomp estim}
\sum_{t\in\Gamma} \int_{U_t}|g_t(x)|^q \nu^{-q}(x) \dx \leq C \int_{\Omega} |g(x)|^q \omega^{-q}(x)\dx.
\end{equation}
\end{theorem}

\begin{proof}  
Observe that since $g\in L^q(\Omega,\omega^{-q})$, then $g\in L^1(\Omega)$. Indeed, using H\"older inequality and the integrability of $\omega^p$:
 \begin{align*}
    \int_\Omega |g(x)| \dx &= \int_\Omega |g(x)| \omega(x)\omega^{-1}(x) \dx \\
    &\le \Big(\int_\Omega |g(x)|^q\omega^{-q}(x)\Big)^\frac{1}{q} \Big(\int_\Omega \omega(x)^p\Big)^\frac{1}{p} \le C\|g\|_{L^q(\Omega,\omega^{-q})}.
\end{align*}  

Now, let $\{\phi_t\}_{t\in\Gamma}$ be a partition of the unity subordinate to $\{U_t\}_{t\in\Gamma}.$ In other words, we have that ${\rm{supp}}(\phi_t)\subset U_t$, $0\le\phi_t(x)\le 1$ and $\sum_t \phi_t(x) = 1,\quad \forall x\in\Omega$. Now, we can define an initial decomposition for $g$ given by $f_t = g\phi_t$. The collection $\{f_t\}_{t\in\Gamma}$ satisfies properties $(1)$ and $(2)$ in Definition \ref{Definition decomposition}, but not necessarily $(3)$. Hence, we modify these functions in order to obtain the $\mathcal{C}$-orthogonality. 

We define, for $s\in \Gamma$, the \emph{shadow} $W_s$ of $U_s$, as:
\[W_s = \bigcup_{k\succeq s}U_k,\]
and for $s\neq a$
\[h_s(x) = \frac{\chi_s(x)}{|B_s|}\int_{W_s} \sum_{k\succeq s} f_k,\]
where $\chi_s(x)$ is the characteristic function of $B_s$.
Note that ${\rm{supp}}(h_s)\subset B_s$ and $\int h_s(x) dx = \int_{W_s}\sum_{k\succeq s} f_k$. Now, we take: 
\[g_t(x) = f_t(x) + \Big(\sum_{s:s_p = t}h_s(x)\Big) - h_t(x) \quad \forall t\neq a,\]
\[g_a(x) = f_a(x) + \Big(\sum_{s:s_p = a}h_s(x)\Big).\]
Note that the summations above are finite since they are indexed over the children of $t$ (or $a$). With this definitions, we have for $t\neq a$: 
 \begin{align*}
\int_\Omega g_t &= \int_{U_t} f_t + \sum_{s:s_p=t}\int_{B_s} h_s - \int_{B_t} h_t \\
 &= \int_{U_t} f_t + \sum_{s:s_p=t}\int_{W_s}\sum_{k\succeq s}f_k-\int_{W_t}\sum_{k\succeq t}f_k \\ 
&= \int_{U_t} f_t + \sum_{s:s_p=t}\sum_{k\succeq s}\int_{U_k}f_k - \sum_{k\succeq t}\int_{U_k} f_k \\
&= \sum_{k\succeq t} \int_{U_k} f_k- \sum_{k\succeq t}\int_{U_k} f_k= 0.
\end{align*} 
Whereas for $t=a$: 
 \begin{align*}
\int_\Omega g_a &= \int_{U_a} f_a + \sum_{s:s_p=a}\int_{B_s} h_s \\
 &= \int_{U_a} f_a + \sum_{s:s_p=a}\int_{W_s}\sum_{k\succeq s}f_k \\ 
&=  \int_\Omega \sum_{k\succeq a} f_k 
=\int_\Omega g= 0.
\end{align*} 

Hence, $\{g_t\}_{t\in\Gamma}$ is a $\mathcal{C}$-$orthogonal$ decomposition of $g$. It remains to prove estimate \eqref{Decomp estim}, which is a consequence of inequality \eqref{dHardy dual}. 
 Recall that the support of each $h_s$, with $s\in \Gamma^*$, is included in $B_s$, and the collection of open sets $\{B_t\}_{t\neq a}$ is pairwise disjoint. Moreover, $h_s$ appears in the definition of $g_t$ if and only if $t=s$ or $t=s_p$. Now we can prove the estimate:
 \begin{align*}
&\sum_{t\in\Gamma} \int_{U_t}|g_t(x)|^q\nu(x)^{-q}\dx & \\
& \le \sum_{t\in\Gamma^*} 2^{q-1}  \int_{U_t}\left(|f_t(x)|^q + \left|\Big(\sum_{s:s_p = t}h_s(x)\Big) - h_t(x)\right|^q\right)\nu^{-q}(x)\dx  \\
& \left. \quad\quad\quad  + 2^{q-1}\int_{U_a}\left(|f_a(x)|^q + \left|\sum_{s:s_p = a}h_s(x)\right|^q\right)\nu^{-q}(x)\dx\right.\\
&= 2^{q-1}\left\{\sum_{t\in\Gamma} \int_{U_t}|f_t(x)|^q\nu^{-q}(x)\dx\right. \\ 
&\left. \quad\quad\quad  + \sum_{t\in\Gamma^*}\int_{U_t}\Big(|h_t(x)|^q+\sum_{s:s_p=t}|h_s(x)|^q\Big)\nu^{-q}(x)\dx\right.\\
&\left.\quad\quad\quad +\int_{U_a}\sum_{s:s_p=a}|h_s(x)|^q \nu^{-q}(x) \dx \right\}\\
&\le 2^{q-1}\sum_{t\in\Gamma} \int_{U_t}|f_t(x)|^q\nu^{-q}(x)\dx +2^{q}\sum_{t\in \Gamma^*}\int_{\Omega}|h_t(x)|^q\nu^{-q}(x)\dx\\
&\le  2^{q-1} \int_{\Omega}|g(x)|^q\nu^{-q}(x)\dx + 2^{q}\sum_{t\in \Gamma^*}\int_{\Omega}|h_t(x)|^q\nu^{-q}(x)\dx \\
&= (I)+(II).
\end{align*} 

The term $(I)$ gives the desired estimate thanks to the embedding   $L^q(\O,\omega^{-q})\hookrightarrow L^q(\O,\nu^{-q})$.  

Now, observe that
\[|h_s(x)|\le \frac{1}{|B_s|}\int_{W_s}\sum_{k\succeq s} |f_k| \le \frac{1}{|B_s|}\int_{W_s}|g|.\] 
Thus, 
 \begin{align*}
(II)& \le C\sum_{t\in\Gamma^*}\nu_t^{-q}|B_t|^{1-q}\Big(\int_{W_t}|g(x)|\dx\Big)^{q} 
\\ &\le  C\sum_{t\in\Gamma^*}\nu_t^{-q}|B_t|^{1-q}\Big(\sum_{s\succeq t}\int_{U_s}|g(x)|\dx\Big)^q
\end{align*} 
Now, we apply \eqref{dHardy dual} with $b_s=\int_{U_s}|g|$, obtaining: 
 \begin{align*}
(II)&\le C\sum_{t\in\Gamma^*}\omega_t^{-q}|B_t|^{1-q}\Big(\int_{U_t}|g(x)|\dx\Big)^{q} \\ 
&\le C\sum_{t\in\Gamma^*}\omega_t^{-q}|B_t|^{1-q}\Big(\int_{U_t} \omega(x)^{p}\dx\Big)^{\frac{q}{p}}\Big(\int_{U_t}|g(x)|^q\omega^{-q}(x)\dx\Big)\\
&\le C\sum_{t\in\Gamma^*} \omega_t^{-q}|B_t|^{1-q}\omega_t^{q}|U_t|^{\frac{q}{p}}\Big(\int_{U_t}|g(x)|^q\omega^{-q}(x)\dx\Big) \\
&\le C\sum_{t\in\Gamma^*}\int_{U_t}|g(x)|^q\omega^{-q}(x)\dx,
\end{align*} 
which completes the proof. The constant $C$ depends on the constants in \eqref{admissible} for $\omega$ and $\nu$, on the constant $N$ that appears on Definition \ref{Decomp of Omega}, on $\sup_s\frac{|U_s|}{|B_s|}$ and linearly on the constant on \eqref{dHardy dual}.
\qed\end{proof}

\begin{remark}\label{rmk integrability cont} 
Observe that, combining Remark \ref{rmk weights} with Remark \ref{rmk integrability}, it is easy to check that the requirement $\omega^p\in L^1(\Omega)$ is implied by the validity of the Hardy-type inequality \eqref{dHardy dual}.
\end{remark}

\section{Decomposition on H\"older domains}
\label{Holder}
In this section, we prove that a $\mathcal{C}$-orthogonal decomposition of a function $g$ as the one given in Theorem \eqref{Decomp Thm} can be obtained when $\Omega$ is a H\"older-$\alpha$ domain, and the weights are powers of the distance to $\partial\Omega$. 

Let $\Omega$ be a bounded domain whose boundary is locally the graph of a function $\varphi$ that verifies: $|\varphi(x)-\varphi(y)|\le K_\varphi|x-y|^\alpha$ for all $x,y$. Our approach follows the construction given in \cite[Section 6]{L1}.

Let $\varphi:(-\frac{3\ell}{2},\frac{3\ell}{2})^{n-1}\to \R$ be a H\"older-$\alpha$ function with $0<\alpha\le 1$ and $\ell>0$. We also assume that $2\ell\le \varphi<3\ell$. Consider:
\begin{equation}\label{def Omega phi}
    \Omega_\varphi = \Big\{(x',x_n)\in \big(-\tfrac{\ell}{2},\tfrac{\ell}{2}\big)^{n-1}\times \R,\;0<x_n<\varphi(x')\Big\}.
\end{equation}
We could assume $\Omega$ is locally $\Omega_\varphi$, but in that case, the distance to $\partial\Omega$ is not necessarily equivalent to the distance to the portion of the graph of $\varphi$ above $(-\frac{\ell}{2},\frac{\ell}{2})^{n-1}.$ Thus, in order to solve this problem, we assume $\Omega$ is locally an expanded version of $\Omega_\varphi$:
\begin{equation}\label{def Omega phi E}
\Omega_{\varphi,E} = \Big\{(x',x_n)\in \big(-\tfrac{3\ell}{2},\tfrac{3\ell}{2}\big)^{n-1}\times \R,\;0<x_n<\varphi(x')\Big\}.
\end{equation}
Now, for $x\in \Omega_\varphi$, the distance to $\partial\Omega$ is equivalent to the distance to:
\[G = \Big\{(x',x_n)\in \big(-\tfrac{3\ell}{2},\tfrac{3\ell}{2}\big)^{n-1}\times \R,\;x_n=\varphi(x')\Big\}.\]

We denote $d_G$ the distance to $G$. Now, we can prove our first result regarding H\"older-$\alpha$ domains, namely:

\begin{lemma}\label{lemma Holder}
Let $\Omega_\varphi$ be the domain defined  in \eqref{def Omega phi} for some $0<\alpha\le 1$, and $\beta$ satisfying:
\begin{equation}\label{cond beta}
\beta p>-\alpha.
\end{equation}
Then, given $f$ in $L^q(\Omega_\varphi,d_G^{-\beta q})$ with vanishing mean value, there exists a $\mathcal{C}-$decomposition of $f$ that satisfies estimate \eqref{Decomp estim} with $\omega = d_G^{\beta}$ and $\nu=d_G^{\beta+\alpha-1}$.
\end{lemma}
\begin{proof}
We build a tree covering of $\Omega_\varphi$ and prove that Theorem \ref{Decomp Thm} holds on it. The main idea is to give a Whitney-type decomposition of  $\Omega_\varphi$ into cubes that satisfy:
\begin{itemize}
\item The edge $\ell_t$ of a cube $Q_t$ is proportional to $d_G(Q_t)$.
\item Two adjacent cubes have comparable sizes. 
\end{itemize}
The cubes are constructed level by level, moving upward towards the graph of $\varphi$. The level $0$ is given by the root cube $Q_a=(-\tfrac{\ell}{2},\tfrac{\ell}{2})^{n-1}\times(0,\ell)$. The other cubes are built recursively. Suppose that $Q_t=Q_t'\times(x_{n,t}^1,x_{n,t}^2)$ is a cube of level $m$. Then, cubes $Q_s$ in level $m+1$, with $s_p=t$, are defined in the following way: consider the cube $Q = 3(Q_t + (0,\dots,0,\ell_t))$, which denotes an expansion of a translated copy of $Q_t$. Then: 
\begin{itemize}
  \item If $Q\subset\Omega_{\varphi,E}$, then we define only one cube $Q_s$ at level $m+1$ with $s_p=t$, $Q_s=Q_t+(0,\dots,0,\ell_t)$.
  \item If $Q\nsubset \Omega_{\varphi,E}$, we define $2^{n-1}$ cubes $Q_s$ at level $m+1$ with $s_p=t$, written as $Q_s = \widetilde{Q}_{t}'\times(x_{n,t}^2,x_{n,t}^2+\tfrac{\ell_t}{2})$, where $\widetilde{Q}_{t}'$ is one of the $(n-1)$-dimensional cubes given by the partition of $Q_t'$ into $2^{n-1}$ cubes with edges of length $\tfrac{\ell_t}{2}$. 
\end{itemize}
See Figure \ref{fig tree} for an example of this construction. 
\begin{figure}[!ht]
\includegraphics[scale=0.4]{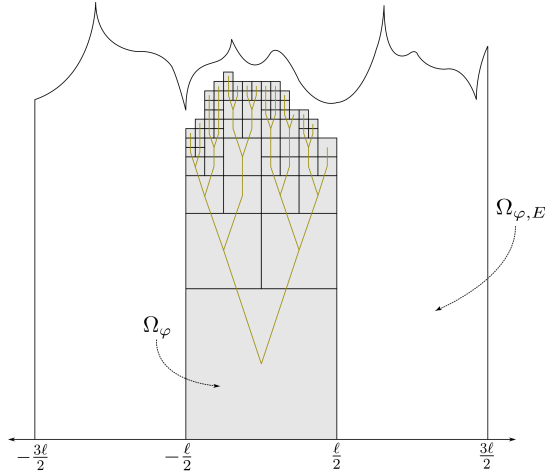}
\caption{Partial representation of a partition of $\Omega_\varphi$ into cubes.}\label{fig tree}
\end{figure}
It is easy to check that this partition satisfies the two main properties of a Whitney decomposition mentioned above. Recall that in a tree covering a certain overlapping of the elements is needed, but our cubes are pairwise disjoint, so we need to enlarge them. If $Q_t= Q_t'\times (x_{n,t}^1,x_{n,t}^2)$, we can expand it downward with a half of itself, defining $U_t = Q_t'\times(x_{n,t}^1-\tfrac{\ell_t}{2},x_{n,t}^2)$. Now $\{U_t\}_t$ is a tree covering, with $B_t = Q_t'\times(x_{n,t}
^1-\frac{\ell_t}{2},x_{n,t}^1)$. We denote $\Gamma$ the underlying set of indices with tree structure. 

Now, we need to prove that Theorem \eqref{Decomp Thm} holds for the weights $ \omega=d_G^\beta$ and $\nu=d_G^{\beta+\alpha-1}$.  Notice that the tree covering defined above satisfies that $\frac{|U_t|}{|B_t|} = 3/2$ for every $t\in\Gamma^*$. Moreover, observe that $L^q(\Omega_{\varphi},d_G^{-(\beta+\alpha-1)})\hookrightarrow L^q(\Omega_{\varphi  q},d_G^{-\beta q})$ since $\beta\geq \beta+\alpha-1$. Notice that, by construction:
\[\max d_G(z) \le C\min d_G(z) \sim \ell_t \quad \forall z\in Q_t,\]
which implies that the weights $\omega$ and $\nu$ are admissible.
Thus, it is enough to prove the Hardy-type inequality \eqref{dHardy dual} and the integrability of the weight $\omega^{p}$. 

As we mentioned in Remark \ref{rmk weights}, \eqref{dHardy dual} is equivalent to \eqref{dHardy dual general} and \eqref{dHardy primal} with
$u_t=|B_t|^\frac{1}{p}\ell_t^{\beta+\alpha-1}$ and $v_t=|B_t|^\frac{1}{p}\ell_t^\beta$. The rest of the proof is devoted to verifying the sufficient condition \eqref{tree cond} for these weights.

Without loss of generality, we assume that $\ell=1$, and thus the edge of every cube is $2^{-j}$ for some $j\in\mathbb{N}$. Since \eqref{tree cond} involves summations over the shadow of a node ($\mathcal{S}_t$), and over the path that goes from $a$ to $t$ ($\mathcal{P}_t$), we begin by estimating the number of cubes of a given size both in $\mathcal{S}_t$ and $\mathcal{P}_t$.

Consider a cube $Q_t$, and take $x'_t\in Q_t'$. Let us take $Q_s$ the first cube going backwards from $Q_t$, such that $s\preceq t$ and $\ell_s = 2\ell_t$. We have that $3(Q_s + (0,\dots,0,\ell_s))\nsubset\Omega_{\varphi_E}$. Hence, there is some $x'_s\in 3Q_s'$ such that $\varphi(x'_s)\le x_{n,s}^2+2\ell_s\le x_{n,t}^1+2\ell_s=x_{n,t}^1+4\ell_t$ (see Figure \ref{fig cubes} left). Now, for every $x'_t\in Q_t'$:
 \begin{align*}
    |\varphi(x'_t)|&\le |\varphi(x'_t)-\varphi(x'_s)| + |\varphi(x'_s)|\le K_\varphi|x'_t-x'_s|^\alpha + x_{n,t}^1+4\ell_t \\
    &\le C_{\alpha}K_\varphi \ell_t^\alpha + 4\ell_{t}+x_{n,t}^1.
\end{align*} 
Now, let us consider $W_t=\bigcup_{k\succeq t}U_k$, the union of all the cubes in the shadow of $U_t$ (which we also called \emph{shadow}). Then, the above estimate gives: 
\[|W_t|\le \ell_t^{(n-1)}(C_{\alpha}K_\varphi \ell_t^\alpha + 4\ell_{t}) \le C_{n,\alpha} \ell_t^{n-1+\alpha}(K_\varphi +\ell_t^{1-\alpha}).\]

Finally, for $k\in\mathbb{N}$, let us denote $\mathbb{P}_i(t)$ and $\mathbb{W}_i(t)$ the number of cubes of size $2^{-i}$ in $\mathcal{P}_t$ and $\mathcal{S}_t$ respectively. Namely:
 \begin{align*}
\mathbb{P}_i(t) &= \#\{r\in\Gamma:\; r\preceq t, \,\ell_r=2^{-i}\}\\
\mathbb{W}_i(t) &= \#\{r\in\Gamma:\; r\succeq t, \,\ell_r=2^{-i}\}
\end{align*} 
We want to estimate both of these quantities. For $\mathbb{P}_i(t)$, we can take $r\in\Gamma$ the lowest index in $\mathcal{P}_t$ such that $\ell_r = 2^{-i}$, and consider $W_r$. $\mathbb{P}_i(t)$ is at most the number of cubes with edges $\ell_r$ in $W_r$. Hence:
\[\mathbb{P}_i(t) \le \frac{|W_r|}{|Q_r|} \le C_{n,\alpha} \ell_r^{n-1+\alpha}(K_\varphi +\ell_r^{1-\alpha})\ell_r^{-n} \le C \ell_r^{-1+\alpha} = C2^{i(1-\alpha)}.\]
Observe that, in particular, this is an estimate for the number of cubes of the same size in a chain of cubes.  
\begin{figure}\label{fig cubes}
\includegraphics[scale=0.5]{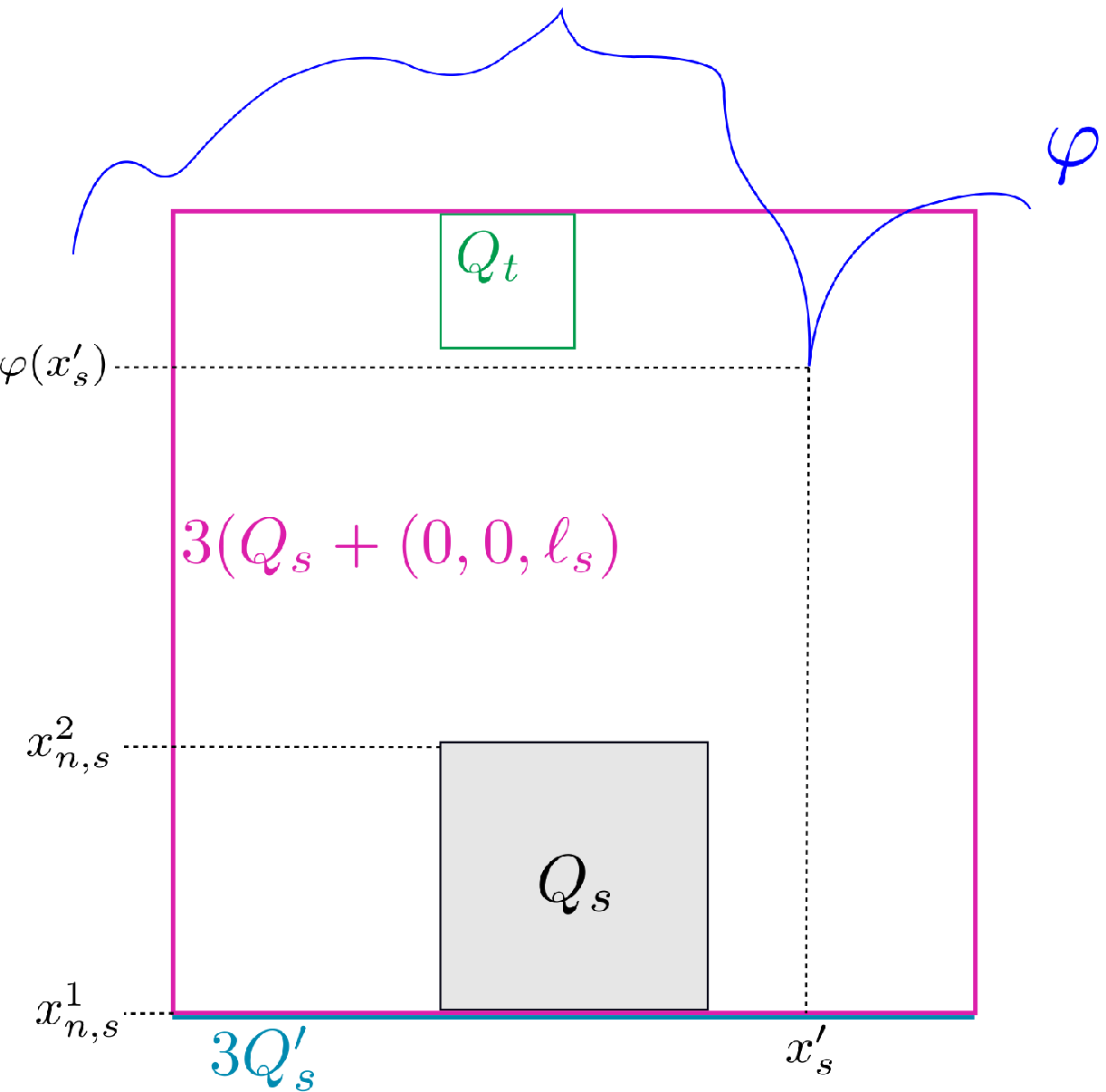} \hspace{2cm}
\includegraphics[scale=0.5]{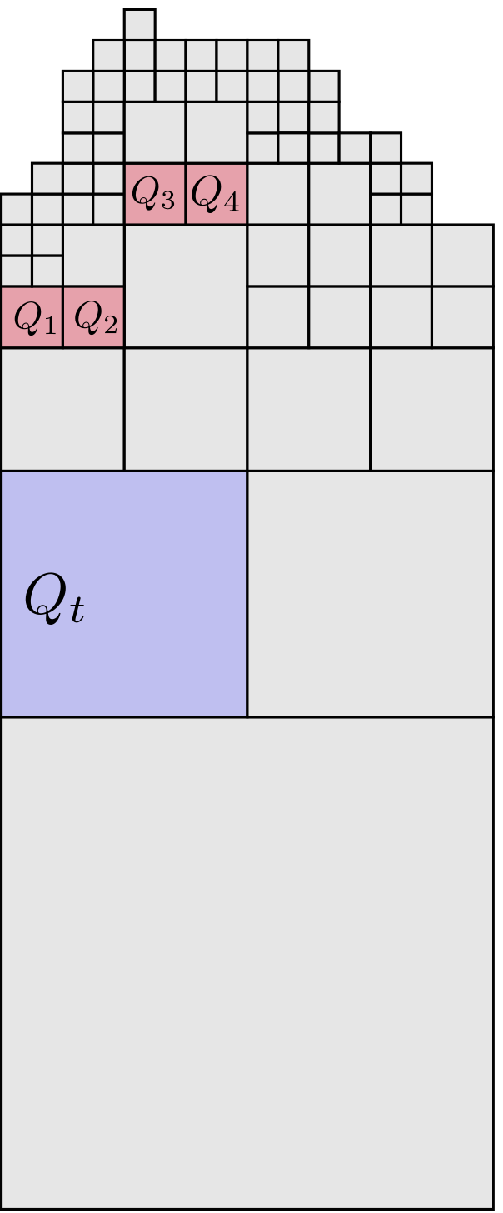}
\caption{Left: given $Q_t$, we can estimate $\varphi(x'_t)$ in terms of $\ell_t$. Right: a cube $Q_t$ and the first cubes with edges of length $2^{-k}$ in its shadow.}
\end{figure}
On the other hand, for $\mathbb{W}_i(t)$, we assume $\ell(t) = 2^{-k}$, and consider the set of the first cubes $Q_r$ such that $r\succeq t$ and $\ell_r=2^{-i}$. In Figure \ref{fig cubes} (right)  a cube $Q_t$ is shown, along with the four first cubes of a certain size in $W_t$. There are $\frac{\ell_t^{n-1}}{2^{-i(n-1)}}$ of such cubes. Moreover each of these cubes can be followed by a chain containing at most: $2^{-i(-1+\alpha)}$ cubes. Therefore:
\[\mathbb{W}_i(t) \le \ell_t^{n-1} 2^{-i(-1+\alpha-n+1)} =\ell_t^{n-1} 2^{-i(\alpha-n)} = 2^{-k(n-1)-i(\alpha-n)}.\]

Finally, we can prove sufficient condition \eqref{tree cond}. We have three indices in $\Gamma$: $r$, $s$ and $t$, for which we assume: $\ell_r=2^{-i}$, $\ell_s=2^{-j}$, $\ell_t=2^{-k}$. Hence:
 \begin{align*}
\sum_{a\prec r\preceq t} u_r^{-q} &= \sum_{a\prec r\preceq t}|B_r|^{-\frac{q}{p}}d_G^{-q(\beta+\alpha-1)} = \sum_{i=0}^k \mathbb{P}_i(t) (2^{-in})^{-\frac{q}{p}} 2^{iq(\beta+\alpha-1)} \\
&\le \sum_{i=0}^k 2^{i(1-\alpha+n\frac{q}{p}+q(\beta+\alpha-1))}\le C 2^{kq(\frac{n-1+\alpha}{p}+\beta)}
\end{align*} 
In the last step we used that the exponent is positive. Indeed: 
\[\frac{n-1+\alpha}{p}+\beta = \beta+\frac{\alpha}{p} + \frac{n-1}{p}>0,\]
since $\beta p> -\alpha$. In the same way, we obtain that:
 \begin{align*}
\sum_{a\prec r\preceq s} u_r^{-q}&\le C 
2^{jq(\frac{n-1+\alpha}{p}+\beta)}
\end{align*} 
Let us denote:
\[I_t = \Big(\sum_{a\prec s\preceq t}u_s^{-q}\Big)^\frac{1}{\theta q}\Big(\sum_{s\succeq t} v_s^p\Big(\sum_{a\prec r\preceq s}u_r^{-q}\Big)^{\frac{p}{q}(1-\frac{1}{\theta})}\Big)^\frac{1}{p}\]
Then:
 \begin{align*}
I_t &\le 2^{k(\frac{n-1+\alpha}{p}+\beta)\frac{1}{\theta}} 
\Big(\sum_{s\succeq t}|B_s|d_G^{\beta p}\Big(\sum_{a\prec r\preceq s} u_r^{-q}\Big)^{\frac{p}{q}(1-\frac{1}{\theta})}\Big)^\frac{1}{p} \\
&\le 2^{k(\frac{n-1+\alpha}{p}+\beta)\frac{1}{\theta}}  \Big(\sum_{j=k}^\infty\mathbb{W}_j(t) 2^{-jn}2^{-j\beta p}2^{j(\frac{n-1+\alpha}{p}+\beta)p(1-\frac{1}{\theta})}\Big)^\frac{1}{p}\\
&\le 2^{k(\frac{n-1+\alpha}{p}+\beta)\frac{1}{\theta}} 
\Big(\sum_{j=k}^\infty 2^{-k(n-1)-j(\alpha-n)} 2^{-j(n+\beta p-(\frac{n-1+\alpha}{p}+\beta)p(1-\frac{1}{\theta}))}\Big)^\frac{1}{p} \\
&\le 2^{k(\frac{n-1+\alpha}{p}+\beta)\frac{1}{\theta}}  2^{-k(n-1)\frac{1}{p}}\Big(\sum_{j=k}^\infty 2^{-j(\alpha+\beta p-(\frac{n-1+\alpha}{p}+\beta)p(1-\frac{1}{\theta}))}\Big)^\frac{1}{p}
\end{align*} 
For the summation to be finite, we need the exponent of $2$ to be negative or equivalently:
 \begin{align*}
\alpha+\beta p-\Big(\frac{n-1+\alpha}{p}+\beta\Big)p\Big(1-\frac{1}{\theta}\Big)&>0\\
\end{align*} 
But \eqref{cond beta} implies $\alpha+\beta p>0$, so we can choose $\theta>1$ such that inequality above remains valid. Now, we can continue the estimate:
 \begin{align*}
    I_t &\le 2^{k(\frac{n-1+\alpha}{p}+\beta)\frac{1}{\theta}}  2^{-k(n-1)\frac{1}{p}} 2^{-k(\alpha+\beta p-(\frac{n-1+\alpha}{p}+\beta)p(1-\frac{1}{\theta}))\frac{1}{p}}\\
    &\le 2^{k[(\frac{n-1+\alpha}{p}+\beta)\frac{1}{\theta}-\frac{n-1}{p}-\frac{\alpha}{p}-\beta+(\frac{n-1+\alpha}{p}+\beta)(1-\frac{1}{\theta}) ]}
\end{align*} 
Since $A_{tree}$ is the supremum of $I_t$ over $t$, we need $I_t$ to be bounded uniformly on $t$, which is to say on $k$, hence we need:
\[
E = \Big(\frac{n-1+\alpha}{p}+\beta\Big)\frac{1}{\theta}-\frac{n-1+\alpha}{p}-\beta+\Big(\frac{n-1+\alpha}{p}+\beta\Big)\Big(1-\frac{1}{\theta}\Big)\le0
\]
But it is easy to check that, in fact, $E=0$. 

Finally, as we mentioned in Remark \ref{rmk integrability cont} the integrability of the weight $\omega^p$ is implied by the sufficient condition for the Hardy-type inequality \eqref{dHardy dual}. Indeed:
 \begin{align*}
\int_\Omega \omega^p(x) \dx &\le \sum_{t\in\Gamma}\int_{U_t}\omega^p(x)\dx 
\le \sum_{t\in\Gamma} |U_t|\omega_t^p 
\le C\sum_{k=0}^\infty \mathbb{W}_k(a) 2^{-kn}2^{-k\beta p} \\
&\le C\sum_{k=0}^\infty 2^{-k(\alpha-n)}2^{-kn}2^{-k\beta p} 
= C\sum_{k=0}^\infty 2^{-k(\alpha+\beta p)}
\end{align*} 
which is finite, since $\beta p > -\alpha$. 
\qed\end{proof}

\begin{remark}
The previous lemma was stated assuming a certain fixed shift in the exponents of the weights. However, one can prefer to consider the general case, with two different weights $\nu = d_G^{\gamma}$ and $\omega = d_G^{\beta}$. In that case, the proof of Lemma \ref{lemma Holder} can be reproduced verbatim until the last step, where the exponent $E$ in the estimate of $I_t$ should be studied. The requirement $E\le 0$ is not automatically fulfilled, but implies the restriction $\gamma\le \beta+\alpha-1$, where the natural shift between the weights becomes apparent. In order to simplify the proof, we stated the lemma in the critical case $\gamma=\beta+\alpha-1$, which is the most useful.
\end{remark}

Lemma \ref{lemma Holder} constitutes the core of the decomposition on H\"older-$\alpha$ domains. In order to extend this result to a complete H\"older domain, we just need to cover it with patches given by rectangles of the form of $\Omega_\varphi$: 

\begin{theorem}\label{teo Holder decomp}
Let $\Omega$ be a H\"older-$\alpha$ bounded domain, with $0<\alpha\le 1$, and $\beta$ satisfying that $\beta p>-\alpha$. Then, given $g\in L^q(\Omega,d^{-\beta q})$, with vanishing mean value, there exists a {$\mathcal{C}-$decomposition} of $g$ subordinate to a partition $\{U_t\}_{t\in\Gamma}$ of $\Omega$ that satisfies:
\begin{equation}\label{decomp estim Holder}
\sum_{t\in\Gamma} \int_{U_t}|g_t(x)|^q d^{(-\beta-\alpha+1) q }(x) \dx \leq C \int_{\Omega} |g(x)|^q d^{-\beta q}(x)\dx.
\end{equation}
In addition, the partition is formed by one smooth domain $\Omega_0$, with positive distance to $\partial\Omega$, and denumerable cubes or cubes extended by a factor $3/2$ in one direction.  

\end{theorem}
\begin{proof}
Let us begin by covering $\partial\Omega$ with a finite number of open sets $\{\mathcal{U}_i\}_{i}$ for $i=1,\dots,m$, such that $\Omega_i = \mathcal{U}_i\cap \Omega$ are of the form of $\Omega_\varphi$, defined in  \eqref{def Omega phi}. We may assume also that there are open sets $\mathcal{V}_i$ such that $\mathcal{V}_i\cap\Omega$ are of the form of $\Omega_{\varphi,E}$, defined in \eqref{def Omega phi E}. Then, we take a smooth domain $\Omega_0$ that intersects each $\Omega_i$, with $1\leq i\leq m$, and such that $d(\Omega_0,\partial\Omega)\geq\delta>0$, and $\bigcup_{i=0}^m \Omega_i=\Omega$.

We continue by using the idea by Bogovskii in \cite{Bogovskii} for a finite partition and Lemma \ref{lemma Holder} in each $\Omega_i$, with $1\leq i\leq m$. Indeed, let us apply an inductive argument: given two sets $A,B\subset\Omega$ such that $|A\cap B|>0$ and a function $f\in L^q(\Omega,d^{-\beta q})$ such that $\int_{A\cup B}f=0$. Then, we can decompose $f$ in $A\cup B$ in the following way:
\[f(x) = \underbrace{\chi_A(x)f(x)-\frac{\chi_{A\cap B}(x)}{|A\cap B|}\int_A f}_{f_A(x)} 
        + \underbrace{\frac{\chi_{A\cap B}(x)}{|A\cap B|}\int_B f+\chi_{B\setminus A}(x) f(x)}_{f_B(x)}.\]
From the integrability of $d^{\beta p}$, the functions $f_A$ and $f_B$ are well-defined and have finite norms in $L^q(\Omega,d^{-\beta q})$. Thus, there is a constant $C$ such that:
\[\|f_A\|_{L^q(\Omega,d^{-\beta q})}+\|f_B\|_{L^q(\Omega,d^{-\beta q})}\le C \|f\|_{L^q(\Omega,d^{-\beta q})}.\] Furthermore, it is easy to check that $f_A$ and $f_B$ are supported in $A$ and $B$ respectively and that both has vanishing mean value.

Next, we can apply this argument with $A= \Omega_m$ and $B=\bigcup_{i=0}^{m-1}\Omega_i$, and then again with $A=\Omega_{m-1}$ and $B=\bigcup_{i=0}^{m-2}\Omega_i$, etc. Therefore, we obtain for every $g\in L^q(\Omega,d^{-\beta q})$ with vanishing mean value on $\Omega$, a decomposition:
$g = \sum_{i=0}^m g_i$, such that $g_i$ is supported on $\Omega_i$ and has vanishing mean value, with the estimate \[\sum_{i=0}^m \|g_i \|^q_{L^q(\Omega,d^{-\beta q})} \le C\|g\|^q_{L^q(\Omega,d^{-\beta q})}.\]  

Now, each $g_i$ for $i=1,\dots,m$ can be decomposed applying Lemma \ref{lemma Holder} with estimate \eqref{Decomp estim}. And, using that $d(\Omega_0,\partial\Omega)\geq\delta>0$, we have:
\[\|g_0 \|_{L^q(\Omega,d^{(-\beta-\alpha+1)q})} \le C\|g_0\|_{L^q(\Omega,d^{-\beta q})},\]
which completes the proof.
\qed\end{proof}

\section{Applications to inequalities on H\"older domains}
\label{Applications}
\setcounter{equation}{0}

In this section we present several results regarding different inequalities on H\"older-$\alpha$ domains. In all the cases the proof follows a similar model: given a function $f$ with vanishing mean value on $\Omega$, we consider a partition $\{U_t\}_{t\in\Gamma}$, as the one provided by Theorem \ref{teo Holder decomp} and apply the decomposition to $f$. Then, we apply an unweighted version of the inequality on each $U_t$, for $t\in\Gamma$, and take advantage of the estimate \eqref{decomp estim Holder} for recovering a global norm. For doing this we rely heavily on the fact that the distance to $\partial\Omega$, $d(x)$, can be regarded as constant over each $U_t$. In other words, we can define values $d_t$ such that $d_t\sim d(x),\,\forall x\in U_t$, where the constants involved in the proportionality are independent of $t$.  Moreover, we have that each $U_t$ is either a smooth domain ($\Omega_0$ in the proof of Theorem \ref{teo Holder decomp}), or a cube or a cube expanded along one direction by a factor $\frac{3}{2}$. For this simple domains, we can control the constant involved in the unweighted inequality.

The divergence problem is solved \emph{directly}: we apply the decomposition to the data $f$. For the other results a duality characterization of the norm on the left hand side is used, and the decomposition is applied to the function in the dual space of the one where the function involved in the inequality belongs. For applying this argument we need the lemma below. 

Recall that the weight $d^{\beta p}$ is integrable over $\Omega$. Thus, let us define the following subspace of $L^q(\Omega,d^{-\beta q})$:
 \begin{align*}V:=\Big\{&g(x) +\psi d^{\beta p}(x)\colon \\  & g(x)\in L^q(\Omega,d^{-\beta q}) \text{ and } \psi\in\R, \text{with }\, \overline{\supp(g)}\subset \Omega, \int_\Omega g=0, \Big\}.
\end{align*} 

\begin{lemma}\label{lemma density}
$V$ is dense in $L^q(\Omega,d^{-\beta q})$, and any $g+\psi d^{\beta p}\in V$ verifies that  
\[\|g \|_{L^q(\Omega,d^{-\beta q})} \le 2\|g+\psi d^{\beta p}\|_{L^q(\Omega,d^{-\beta q})}.\]
\end{lemma}
\begin{proof}
First, let us prove the estimation in the lemma. Notice that \[\psi=\frac{\int_\Omega g+\psi d^{\beta p}}{\int_\Omega d^{\beta p}}.\]

Thus, by using the H\"older inequality we obtain 
 \begin{align*}
\|\psi d^{\beta p}\|_{L^q(\Omega,d^{-\beta q})}
&\leq \frac{\left|\int_\Omega g+\psi d^{\beta p}\right|}{\int_\Omega d^{\beta p}}\|d^{\beta p}\|_{L^q(\Omega,d^{-\beta q})}\\ 
&\leq \left(\int_\Omega d^{\beta p}\right)^{1/p+1/q-1} \|g+\psi d^{\beta p}\|_{L^q(\Omega,d^{-\beta q})},
\end{align*} 
which implies the estimate.

Now, given $F\in L^q(\Omega,d^{-\beta q})$ and $\varepsilon > 0$, let us show that there exists $g_F+\psi_F d^{\beta p}$ in $V$ sufficiently close to $F$. Using again that $d^{\beta p}$ is integrable, we define $\psi_F$ by  
\[\psi_F=\frac{\int_\Omega F}{\int_\Omega d^{\beta p}}.\] 
Then, the function $h_F(x)=F(x)-\psi_F d^{\beta p}$ has a vanishing mean value, but it does not necessarily have a compact support. Thus, let $B$ be an open ball, independent of $\varepsilon$, such that $\overline{B}\subset \Omega$. And, let $\Omega_\varepsilon$ be an open set that contains $B$ such that $\overline{\Omega_\epsilon}\subset \Omega$ and 
\[\left\|(1-\chi_{\Omega_\varepsilon}(x))h_F(x)\right\|_{L^q(\Omega,d^{-\beta q})}<\varepsilon,\]
where $\chi$ denotes a characteristic function. Finally, let show that the following function fulfils the requirements
\[g_F(x)=\chi_{\Omega_\varepsilon}(x)h_F(x)+\dfrac{\chi_B(x)d^{\beta p}(x)}{\int_B d^{\beta p}}\int_{\Omega\setminus \Omega_\varepsilon}h_F.\]
Following some straightforward calculations, it can be seen that $g_F(x)+\psi_F d^{\beta p}(x)$ belongs to $V$. And, by using the H\"older inequality multiple times we conclude the proof of the lemma with the following estimation
 \begin{align*}
    &\|F(x)-g_F(x)-\psi_F d^{\beta p}(x)\|_{L^q(\Omega,d^{-\beta q})}\\ =&\|h_F(x)-g_F(x)\|_{L^q(\Omega,d^{-\beta q})}\\ 
    \leq &\left\|(1-\chi_{\Omega_\varepsilon}(x))h_F(x)-\dfrac{\chi_B(x)d^{\beta p}(x)}{\int_B d^{\beta p}}\int_{\Omega\setminus \Omega_\varepsilon}h_F\right\|_{L^q(\Omega,d^{-\beta q})}\\
    \leq &\left(1+\left(\dfrac{\int_\Omega d^{\beta p}}{\int_B d^{\beta p}}\right)^{1/p}\right)\left\|(1-\chi_{\Omega_\varepsilon}(x))h_F(x)\right\|_{L^q(\Omega,d^{-\beta q})}\\ 
    \leq &\left(1+\left(\dfrac{\int_\Omega d^{\beta p}}{\int_B d^{\beta p}}\right)^{1/p}\right) \varepsilon.
\end{align*} 
\qed
\end{proof}

The importance of this lemma will become evident later, in the proof of the improved Poincaré inequality, which is the first result that is obtained via a duality argument. 

\subsection{The divergence equation}\label{divergence}
In this subsection, we study the problem ${\rm div}\,{\bf u} = f$ in $\Omega$ with boundary condition ${\bf u} =\bm{0}$ on $\partial\Omega$, for certain $f$ such that $\int_\Omega f = 0$. In addition we want to obtain an estimate for the norm of the solution ${\bf u}$ in terms of the datum $f$. The unweighted estimate $\|D {\bf u}\|_{L^{q}(\Omega)}\le C\|f\|_{L^q(\Omega)}$, that is valid on regular domains, cannot hold on H\"older-$\alpha$ domains due to the  singularities on the boundary of $\Omega$. Weighted norms can be used to compensate those singularities, as shown in the following inequality:
\[ \|D{\bf u}\|_{L^q(\Omega,d^{(1-\alpha)q})}\le C\|f\|_{L^q(\Omega)}.\]

Such a result was extended in \cite{DLg}, under certain additional hypothesis. Indeed, in that paper only the planar case is considered, and $\partial\Omega$ is assumed to be included in a $1-$Ahlfors regular set. In this context the following estimate is proven:
\begin{equation}\label{weighted div ineq}
\|D{\bf u}\|_{L^q(\Omega,d^{(1-\beta-\alpha)q})}\le C\|f\|_{L^q(\Omega,d^{-q\beta})},
\end{equation}
where the restrictions $0\le\beta\le 1-\alpha$ and $\beta<\frac{1}{q}$ are imposed on $\beta$. It is important to notice that we have stated \eqref{weighted div ineq} in the same terms of our results to simplify the comparison. We follow the same principle when citing previous results in the next subsections.

Observe that the restrictions on $\beta$ allows the weight to be transferred partially (or totally) to the right hand side. The case $\beta=1-\alpha$ is used to prove well-posedness of the Stokes equations. The estimation \eqref{weighted div ineq} was generalized in \cite{L1} where the restrictions on the dimension on $\Omega$, on the parameter $\beta$, and on the Ahlfors regularity on $\partial\Omega$ were lifted, with the exception of the requirement $\beta\geq 0 $. Our result shows that this restriction can be relaxed, and that it is enough to ask $\beta>-\alpha/p$.

\begin{theorem}\label{Divergence in Holder} Let $\Omega\subset\R^n$ be a bounded H\"older-$\alpha$ domain, and $\beta p>-\alpha$. Given $f\in L^q_0(\Omega,d^{-q\beta})$ such that $\int_\Omega f=0$, there exists a vector field ${\bf u}\in W^{1,q}_0(\Omega,d^{q(1-\alpha-\beta)})^n$,
solution of ${\rm div\,}{\bf u}=f$, that verifies the estimate \eqref{weighted div ineq}.
\end{theorem}
\begin{proof}
It is known (see, for example, \cite{ADM}), that given $U$ a John domain and $f\in L^q(U)$ with vanishing mean value, then, there exists ${\bf u}\in W^{1,q}_0(U)^n$ such that $\div{\bf u} = f$ and:
\[\|{\bf u}\|_{W^{1,q}(U)}\le C \|f\|_{L^q(U)}.\]
Moreover, a simple scaling argument shows that the result holds with the same constant $C$ for every cube (or, more generally, for every rectangle with a fixed aspect ratio). Indeed, consider $\hat{Q}=(0,1)^n$, the reference cube. For simplicity, we take $Q\subset \mathbb{R}^n$ some other cube, with edges parallels to the coordinate axis and of length $\ell_Q$. We can consider the affine map $F:\hat{Q}\to Q,$ $F(\hat{x}) = \ell_Q \hat{x}+b,$ being $b$ a fixed vertex of $Q$. Then, given $f\in L^p(Q)$ such that $\int_Q f=0$, we define $\hat{f}(\hat{x})=f(F(\hat{x})),$ and $\hat{{\bf u}}$ the solution of $\div_{\hat{x}}\hat{{\bf u}} = \hat{f}$ on $\hat{Q}$. Now, take ${\bf u}(x) = \ell_Q \hat{{\bf u}}(F^{-1}(x))$. We have $\div_x{\bf u} = \ell_Q \div_x\hat{{\bf u}}(F^{-1}(x)) = \ell_Q \frac{1}{\ell_Q}\div_{\hat{x}} \hat{{\bf  u}}(\hat(x)) = \hat{f}(\hat{x}) = f(x)$. The estimate follows in a similar way, with the constant $C$ being the same for $Q$ as for the fixed cube $\hat{Q}$. If the edges of $Q$ are not parallel to the axis, a rotation needs to be included in $F$, but the same idea follows. 

Now, given $\Omega$ a H\"older-$\alpha$ domain, and $f\in L^q(\Omega,d^{-q\beta})$ such that $\int_\Omega f=0$, we consider $\{f_t\}_{t\in\Gamma}$ the decomposition of $f$ given by Theorem \ref{teo Holder decomp}. For each $f_t$, we have a unique solution ${\bf u}_t$ supported on $U_t$ and such that $\div {\bf u}_t = f_t$ with the unweighted estimate: \[\|D {\bf u}_t\|_{L^q(U_t)}\le c \|f_t\|_{L^p(U_t)}.\]
Since every $U_t$, with the possible exception of the the root of $\Gamma$, is a cube, the constant $c$ can be taken independent of $t$. Now, taking ${\bf u} = \sum_{t\in\Gamma} {\bf u}_t$, it is immediate that $\div {\bf u} = f$. Moreover, we can take a constant  $d_t\sim d(U_t,\partial\Omega)$ for each $t$, and: 
 \begin{align*}
    \|D {\bf u}\|_{L^q(\Omega,d^{(1-\alpha-\beta)q})}^q &\le C \sum_{t\in\Gamma} \|D {\bf u}_t\|_{L^q(U_t,d^{(1-\alpha-\beta)q})}^q\\
    &\le C \sum_{t\in\Gamma}d_t^{(1-\alpha-\beta)q}\|D {\bf u}_t\|_{L^q(U_t)}^q \\
    &\le Cc\sum_{t\in\Gamma}d_t^{(1-\alpha-\beta)q}\|f\|_{L^q(U_t)}^q \\
    &\le C \sum_{t\in\Gamma}\int_{U_t} f(x)^q d(x)^{(1-\alpha-\beta)q}\dx \\
    &\le C \int_\Omega f(x)^q d(x)^{-\beta q} \dx = C\|f\|_{L^q(\Omega,d^{-\beta q})}^q, 
\end{align*} 
where in the last step we used \eqref{decomp estim Holder}. 
\qed\end{proof}

\subsection{Improved Poincar\'e inequality}
Improved Poincar\'e inequalities have been largely studied in several contexts. For a H\"older-$\alpha$ domain $\Omega$, in \cite{BoasStraube} (and later in \cite{DMRT}) it is proven that:
\[\|f\|_{L^p(\Omega)}\le C \| \nabla f\|_{L^p(\Omega,d^{\alpha p})},\]
for every $f$ with vanishing mean value on $\Omega$. 

A weighted extension of this result was given in \cite{ADL}, where the authors proved: 
\[\|f\|_{L^p(\Omega,d^{\beta p})}\le C\|\nabla f\|_{L^p(\Omega,d^{(\beta+\alpha)p})},\]
for $\beta$ satisfying $0\le \beta\le 1-\alpha$. We show that this restrictions on $\beta$ can be reduced to the requirement $\beta>-\alpha/p$.

\begin{theorem}\label{improved Poincare}
Let $\Omega$ be a H\"older-$\alpha$ domain for some $0<\alpha\le 1$, and $f\in L^p(\Omega,d^{\beta p})$ for some $\beta p>-\alpha$, such that $\int_\Omega f d^{\beta p}=0$. Then, there is a constant $C$ such that:
\[\|f\|_{L^p(\Omega,d^{\beta p})}\le C\|\nabla f\|_{L^p(\Omega,d^{(\beta+\alpha)p})}.\]
\end{theorem}
\begin{proof}
We study the norm of $f$ using a duality characterization. Thanks to Lemma \ref{lemma density}, it is enough to consider $h = g+d^{\beta p}\psi\in V$: 
 \begin{align*}
    \|f\|_{L^p(\Omega,d^{\beta p})} &=\sup_{h:\|h\|_{L^q(\Omega,d^{-\beta q})}=1} \int_\Omega f h \\
    &= \sup_{h:\|h\|_{L^q(\Omega,d^{-\beta q})}=1}\int_\Omega f(g+d^{\beta p}\psi)\\
    &=\sup_{h:\|h\|_{L^q(\Omega,d^{-\beta q})}=1}\int_\Omega fg.
\end{align*} 
In the last step, we used that $\int_\Omega f d^{\beta p} = 0$ and $\psi$ is a constant. Now, since $g$ has vanishing mean value, we can apply to it the decomposition of Theorem \eqref{Decomp Thm}:
\[\|f\|_{L^p(\Omega,d^{\beta p})}=\sup_{h:\|h\|_{L^q(\Omega,d^{-\beta q})}=1} = \int_\Omega \sum_{t\in\Gamma}fg_t.\]
Here the necessity of Lemma \ref{lemma density} becomes clear: since the support of $g$ is compact, it intersects only a finite number of sets $U_t$, so the summation is finite and can be pulled out of the integral. Hence, using the $\mathcal{C}-$orthogonality of $g_t$ and the fact that $d(x)\sim d_t$ for $x\in U_t$ we have:
 \begin{align*}
    \int_\Omega fg &= \sum_{t\in\Gamma} \int_{U_t} f g_t = \sum_{t\in\Gamma} \int_{U_t} (f-f_{U_t})g_t\\ 
    &\le \sum_{t\in\Gamma} \|f-f_{U_t}\|_{L^p(U_t,d^{(\beta+\alpha-1)p})}\|g_t\|_{L^q(U_t,d^{(-\beta-\alpha+1)q})}  \\
    &\le \Big(\sum_{t\in\Gamma}\|f-f_{U_t}\|^p_{L^p(U_t,d^{(\beta+\alpha-1)p})})^\frac{1}{p}\Big(\sum_{t\in\Gamma}\|g_t\|^q_{L^q(U_t,d^{(-\beta-\alpha+1)q})}\Big)^\frac{1}{q}\\
    &\le C\Big(\sum_{t\in\Gamma}d_t^{(\beta+\alpha-1)p}\|f-f_{U_t}\|^p_{L^p(U_t)}\Big)^\frac{1}{p}\|g\|_{L^q(\Omega,d^{-\beta q})},
\end{align*} 
where in the last step we used \eqref{decomp estim Holder}.

In order to complete the proof, we recall that, thanks to the estimate in Lemma \ref{lemma density}, $\|g\|_{L^q(\Omega,d^{-\beta q})}\le 2$, and that the Poincar\'e inequality holds on the unweighted case for smooth domains. Moreover, for convex domains the constant is proportional to the the diameter of the domain. In our case, the diameter of each cube $U_t$ is proportional to $d_t$, hence:
 \begin{align*}
    \|f\|_{L^p(\Omega,d^{\beta p})} &\le C %
    \Big(\sum_{t\in\Gamma}d_t^{(\beta-1+\alpha)p}d_t^p\|\nabla f\|_{L^p(U_t)}^p\Big)^\frac{1}{p} \\
    &\le C\Big(\sum_{t\in\Gamma}\|\nabla f\|_{L^p(U_t,d^{(\beta+\alpha)p})}^p\Big)^\frac{1}{p}\\
    &= C \|\nabla f\|_{L^p(\Omega,d^{(\beta+\alpha)p})}
\end{align*} 
\qed\end{proof}

\subsection{Fractional Poincar\'e inequality}

Recently, authors have shown interest in fractional versions of the classical Poincar\'e inequality, for example:

\begin{equation}\label{basic frac Poincare}
\inf_{c\in\R}\|u-c\|_{L^p(U)}\le C\left(\int_U\int_{U\cap B(x,\tau d(x))}\frac{|u(x)-u(y)|^p}{|x-y|^{n+sp}}{\rm d}x{\rm d}y\right)^\frac{1}{p},
\end{equation}
for $\tau\in(0,1)$. The right hand side is similar to the usual seminorm of the fractional Sobolev space $W^{s,p}$ for $0<s<1$ where the double integral is taken over $U\times U$. In fact, both expressions are equivalent for Lipschitz domains  (\cite[equation (13)]{Dyda}). However, if the usual seminorm is taken in \eqref{basic frac Poincare}, it can be seen that the inequality holds for every bounded domain (see, for example \cite[Section 2]{DD}, \cite[Proposition 4.1]{HL}). In particular, it is shown in \cite[Proposition 4.1]{HL} that the constant involved in the inequality is proportional to $diam(U)^{\frac{n}{p}+s}|U|^{-\frac{1}{p}}$. On the other hand, the stronger version \eqref{basic frac Poincare} fails on irregular domains.  Here we prove a weighted improved inequality:

\begin{theorem}
Let $\Omega$ be a H\"older-$\alpha$ domain for some $0<\alpha\le1$, and $u\in W^{s,p}(\Omega,d^{\beta p})$ for some $s\in(0,1)$ and $\beta p> -\alpha$ then, for $\tau\in(0,1)$:
\begin{multline}\label{frac Poincare Holder}
    \inf_{c\in\mathbb{R}}\|u-c\|_{L^p(\Omega,d^{\beta p})}\le \\
    C \left(\int_\Omega\int_{\Omega\cap B(x,\tau d(x))}\frac{|u(y)-u(x)|^p}{|y-x|^{n+sp}}\delta(x,y)^{(s+\beta-\alpha+1)p} {\rm d}x{\rm d}y\right)^\frac{1}{p},
\end{multline}
where $\delta(x,y)=\min\{d(x),d(y)\}.$
\end{theorem}
\begin{proof}
Naturally, it is enough to consider $c=\int_\Omega ud^{\beta p}$. Moreover, we may assume that $\int_{\Omega} u d^{\beta p} = 0.$ 
As usual, writing the norm on the left hand side by duality, via Lemma \ref{lemma density}, applying the decomposition for $g$ and estimate \eqref{decomp estim Holder} yield:
 \begin{align*}
\|u\|_{L^p(\Omega,d^{\beta p})} & =\sum_{t\in\Gamma} \int_{U_t}(u-c_t)g_t \\
&\le \sum_{t\in\Gamma}\|u-c_t\|_{L^p(U_t,d^{p(\beta+\alpha-1)})}\|g_t\|_{L^q(U_t,d^{q(-\beta-\alpha+1)})} \\
&\le C \left(\sum_{t\in\Gamma}\|u-c_t\|_{L^p(U_t)}d_t^{p(\beta+\alpha-1)}\right)^\frac{1}{p},
\end{align*} 
for any set of constants $\{c_t\}_{t\in\Gamma}.$

For completing the proof we invoke \cite[Proposition 4.2]{HL}, that states that for a cube $Q$ with edges of length $\ell(Q)$:
\[\inf_{c}\|u-c\|_{L^p(Q)}\le C_{n,p}\tau^{s-n}\ell(Q)^{s} 
\left(\int_Q\int_{Q\cap B(x,\tau \ell(Q))} \frac{|u(x)-u(y)|^p}{|x-y|^{n+sp}}\right)^\frac{1}{p},\]
for any $\tau\in(0,1)$. We apply this estimate for every cube $U_t$. For the central subdomain $U_0$, we apply \cite[Proposition 4.1]{HL}. Hence, we obtain: 
 \begin{align*}
    \|&u\|_{L^p(\Omega,d^{\beta p})} \\
    &\le C_{n,p}\tau^{s-n}\left(\sum_{t\in\Gamma}\int_{U_t}\int_{U_t\cap B(x,\tau\ell(U_t))}\frac{|u(x)-u(y)|^p}{|x-y|^{n+sp}}\ell(U_t)^{sp}d_t^{p(\beta+\alpha-1)}\right)^\frac{1}{p}.
\end{align*} 
Since $\ell(U_t)\sim d_t\sim d(x)\sim d(y)$ for every $x\in U_t$ and $y\in U_t$, we continue: 
 \begin{align*}
 &\le C_{n,p}\tau^{s-n}\left(\sum_{t\in\Gamma}\int_{U_t}\int_{U_t\cap B(x,\tau d(x))}\frac{|u(x)-u(y)|^p}{|x-y|^{n+sp}}d_t^{p(s+\beta+\alpha-1)}\right)^\frac{1}{p} \\
 &\le C\tau^{s-n}\left(\sum_{t\in\Gamma}\int_{U_t}\int_{ B(x,\tau d(x))}\frac{|u(x)-u(y)|^p}{|x-y|^{n+sp}}\delta(x,y)^{p(s+\beta+\alpha-1)}\right)^\frac{1}{p} \\
 &\le C\tau^{s-n}\left(\int_{\Omega}\int_{ B(x,\tau d(x))}\frac{|u(x)-u(y)|^p}{|x-y|^{n+sp}}\delta(x,y)^{p(s+\beta+\alpha-1)}\right)^\frac{1}{p},
\end{align*} 
which completes the proof.
\qed\end{proof}

This result provides a partial generalization of the one obtained in \cite{DD}. In that paper, a more general form of the inequality is considered, with different exponents $p$ and $q$ on the left and right hand sides, as well as a larger class of domains. However, for technical reasons, when dealing with H\"older-$\alpha$ domains, only the case $\beta = 0$ is considered. Our result is equivalent to \cite[Theorem 5.1]{DD} with $p=q$, but the restriction on $\beta$ is weaker. 

Moreover, \cite[Theorem 5.2]{DD} shows that the shift in the exponent between the left and right hand sides of \eqref{frac Poincare Holder} is optimal. 

\subsection{Korn's inequality}\label{Korn Section}
Given a vector field ${\bf u}\in W^{1,p}(U)^n$,  Korn's inequality states that, 
\begin{equation}\label{basic Korn}\|D{\bf u}\|_{L^p(U)}\le C\|\varepsilon({\bf u})\|_{L^p(U)},
\end{equation}
where $\varepsilon({\bf v}) = \frac{D{\bf v}+D{\bf v}^t}{2}$ is the symmetric gradient. This result fails when $\varepsilon({\bf u})$ vanishes but $D{\bf u}$ does not. Thus, some additional condition on ${\bf u}$ is needed. The so-called \emph{first case} states the inequality when ${\bf u}$ vanishes at the boundary of $U$, and it can be proven using simple arguments, for every bounded domain. We are interested in the \emph{second case} that establishes that \eqref{basic Korn} holds when $\int_U \frac{D{\bf u}-D{\bf u}^t}{2} = 0$. This case requires deeper considerations on the domain and it actually fails for irregular domains. 

A \emph{general case} is also considered in the literature:
\begin{equation}\label{general Korn}
\|D{\bf u}\|_{L^p(U)}\le C\{\|\varepsilon({\bf u})\|_{L^p(U)}+\|{\bf u}\|_{L^p(U)}\},
\end{equation}
which does not need any further assumption on ${\bf u}$. \eqref{general Korn} can be easily derived from the second case of \eqref{basic Korn} (see, for example \cite{BS}). The converse can be proved, for regular domains, using a compactness argument (see \cite{KO})

We prove the following weighted version of \eqref{basic Korn}.

\begin{theorem}\label{Korn}
Let $\Omega$ be a H\"older-$\alpha$ domain for some $0<\alpha\le 1$, and ${\bf u}\in W^{1,p}(\Omega, d^{\beta p})^n$ with $\beta p> -\alpha$, such that $\int_\Omega \frac{D{\bf u}-D{\bf u}^t}{2}d^{\beta p}=0$ then,
\[\|D{\bf u}\|_{L^p(\Omega,d^{\beta p})}\le C \|\varepsilon({\bf u})\|_{L^p(\Omega,d^{(\beta+\alpha-1)p})}.\]
\end{theorem}
\begin{proof}
Observe that if we denote $\eta({\bf u}) = \frac{D{\bf u}-D{\bf u}^t}{2}$, $D{\bf u} = \varepsilon({\bf u})+\eta({\bf u}),$ so it is enough to prove the estimate for the elements $\eta_{i,j}({\bf u})$ of the matrix $\eta({\bf u})$, that have vanishing weighted mean value. The estimate is obtained by following step by step the proof of Theorem \ref{improved Poincare} so we only give references for the needed unweighted inequalities. The norm of $\eta_{i,j}({\bf u})$ is characterized by duality via Lemma \ref{lemma density}. The unweighted estimate \eqref{basic Korn} is known to hold for convex domains with a constant $C$ proportional to the ratio between the diameter of $U$ and the diameter of a maximal ball contained in $U$ (see \cite{D}). Hence, a universal constant can be taken for every cube $U_t$. On the other hand, for the central subdomain $U_0$, we can apply \cite[Corollary 2.2]{DM} where it is shown that \eqref{basic Korn}  holds on domains of Jones, which include Lipschitz domains.
\qed\end{proof}

This generalizes \cite[Theorem 3.1]{ADL} and \cite[Theorem 2.1]{DLg}, where a similar result is proven, but only for $0\le \beta\le 1-\alpha$. In both cases the result is stated in the form of \eqref{general Korn} but it is derived from the second case. In \cite{ADLg} a counterexample is given that shows that the shift $\alpha-1$ between the exponents on the left and right hand sides is optimal. 

\appendix
\section{Proof of Theorem \ref{ehp theorem}}
 
We derive the discrete result from a continuous analogous proven in \cite{EHP}. We begin by obtaining another equivalent form for the Hardy-type inequality: 
\begin{lemma}
Inequalities \eqref{dHardy dual general} and \eqref{dHardy primal} are equivalent (with the same constant $C$) to:
\begin{equation}\label{dHardy primal left}
\left(\sum_{s\in\Gamma^*}\bigg(v_s\sum_{a\prec t\preceq s}F_t u_t^{-1}\bigg)^p\right)^\frac{1}{p} \le C \left(\sum_{s\in\Gamma^*}F_s^p\right)^\frac{1}{p},
\end{equation}
for every ${\bm F} = \{F_t\}_{t\in\Gamma}\in \ell_p(\Gamma)$.
\end{lemma}
\begin{proof}
\eqref{dHardy primal left} is obtained from \eqref{dHardy primal} by changing variables  $F_s=d_s u_s$.
\qed\end{proof}

We derive conditions for \eqref{dHardy primal left} to hold from the continuous case. Let $G_c=(\Gamma_c,E_c)$ be a continuous tree with root $a$. By continuous, we mean that the edges in $E_c$ are segments in the plane, with a certain length. In \cite{EHP}, the authors study the operator $T:L^{p_1}(G)\to L^{p_2}(G)$ given by:
\[T(f) = \eta(x) \int_a^x f(y)\mu(y) dy.\]
Where $\mu$ and $\eta$ are weights and the integral is taken along the path that connects the root $a$ with the point $x$, that could lie anywhere in $G_c$. Here we are only interested in the case $p_1=p_2=p$, so we consider $\mu\in L^{q}_{loc}(G_c)$ and $\eta\in L^p(G_c)$, but the same ideas could be applied to the general case. 
$T$ is continuous in $L^p(G_c)$ if and only if:
\[C_c = \sup_{f\in L^p} \frac{\left(\int_{G_c} \left(\eta(x)\int_a^x f(y) \mu(y) dy\right)^p dx\right)^\frac{1}{p}}{\left(\int_{G_c} f(x)^p dx\right)^\frac{1}{p}}<\infty.\]
Given a sub-tree $K$, we denote $\partial K$ the boundary of $K$. We say that $x\in \partial K$ is maximal if every point $y\succ x$ does not belong to $K$. We define $\mathcal{K}_c$ the set of all sub-trees of $\Gamma$ containing $a$ and such that every boundary point is maximal. We define: 
\[B_c = \sup_{K\in\mathcal{K}_c}\frac{\left(\int_{\Gamma\setminus K}\eta(x)^p dx\right)^\frac{1}{p}}{\alpha_{c,K}},\]
where 
\[\alpha_{c,K}= \inf\Big\{\|f\|_p:\; \int_a^t |f(x)|\mu(x) {\rm d}x,\;\forall t\in\partial K\Big\}\]
Now, we can state the main result of \cite{EHP}, namely:
\begin{theorem}\label{Hardy cont} $T$ is continuous if and only if $B_c<\infty$. Moreover:
\[B_c\le C_c\le 4B_c\]
\end{theorem}
\begin{proof}
See \cite[Theorem 3.1]{EHP}.
\qed\end{proof}

Finally, we can prove the theorem:

\begin{proof}[Proof of Theorem \ref{ehp theorem}]\label{Teo dHardy primal left}
We prove that given $G=(\Gamma,E)$ a discrete tree, ${\bm v}=\{v_t\}_{t\in\Gamma^*}$ and ${\bm u}=\{u_t\}_{t\in\Gamma^*}$ positive weights, inequality \eqref{dHardy primal left} holds for every ${\bf F} = \{F_t\}_{t\in\Gamma^*}$ if and only if: 
\[B = \sup_{K\in \mathcal{S}_a} \frac{\left(\sum_{t\in K'} v_t^p\right)^{\frac{1}{p}}}{\alpha_{K}},\]
Moreover $B\le C\le 4B$ 

In order to apply Theorem \ref{Hardy cont}, we build a continuous tree $G_c$ from $G$ by assigning each edge in $E$ a length of $1$. Take $0<\gve<1$. We denote $(s-\gve,s)$ the points in the edge $(s_p,s)$ that are at a distance less than $\gve$ from $s$. For each $s\in\Gamma^*$, we take $\varphi_\gve^s$ a function such that: 
\[\textrm{supp}(\varphi_\gve^s)\subset(s-\gve,s),\quad\quad \int_{s-\gve}^s\varphi_\gve^s(x)^p {\rm d}x = 1.\]
Now we complete the setting for the continuous problem by defining, for $z\in (s_p,s)$ the functions: 
\begin{itemize}
  \item $f(z) = F_s$,
  \item $\mu(z) = u_s^{-1}$,
  \item $\eta_\gve(z) = v_s \varphi_\gve^s(z)$
\end{itemize}

In the definition of $C_c$ we take $\eta_\gve$ and $\mu$ as the weights and the supremum over all functions that are constant on each edge and denote the result $C_\gve$. Analogously, $B_\gve$ is $B_c$ with $\eta_\gve$ and $\mu$. We prove that $C_\gve \to C$ and $B_\gve\to B$ when $\gve\to 0^+$. 

First, for $C$, observe that we can assume without loss of generality that $F_t$ (hence, $f$) is positive. Then: 
 \begin{align*}
\int_{s_p}^s &\left(\eta_\gve(x) \int_a^x f(y) \mu(y) \dy\right)^p \dx = \int_{s-\gve}^s\left(\eta_\gve(x) \int_a^x f(y) \mu(y) \dy \right)^p \dx \\
&\geq \int_{s-\gve}^s v_s^p \varphi_\gve^s(x)^p \left(\sum_{a\prec t\prec s} F_t u_t^{-1} + (1-\gve) F_s u_s^{-1}\right)^p \dx \\
 &\geq v_s^p \left(\sum_{a\prec t\preceq s} F_t u_t^{-1}\right)^p (1-\gve)\int_{s-\gve}^s\varphi_\gve^s(x)^p \dx\\
 &= (1-\gve) v_s^p \left(\sum_{a\prec t\preceq s} F_t u_t^{-1}\right)^p
\end{align*} 
On the other hand: 
 \begin{align*}
\int_{s_p}^s &\left(\eta_\gve(x) \int_a^x f(y) \mu(y) \dy\right)^p \dx \le \int_{s-\gve}^s \left(\eta_\gve(x) \int_a^s f(y) \mu(y) \dy\right)^p \dx \\
&\le \int_{s-\gve}^s v_s^p \varphi_\gve^s(x)^p\left(\sum_{a\prec t\preceq s} F_t u_t^{-1}\right)^p \dx  
= v_s^p\left(\sum_{a\prec t\preceq s} F_t u_t^{-1}\right)^p
\end{align*} 
Finally, observe that
\[\int_{s_p}^s f(x)^p dx = F_s^p.\]
Hence, we have: 
\[(1-\gve) C\le C_\gve \le C,\]
which proves that $C_\gve\to C$, when $\gve\to 0^+$. 

For $B$, take $\delta>0$ and $K\in\mathcal{K}$ such that 
\[\frac{\left(\sum_{s\in \Gamma\setminus K^\circ} v_s^p\right)^{\frac{1}{p}}}{\alpha_{K}}>B-\delta.\]
Consider $K_\gve$ the continuous subtree obtained by removing from $K$ the points that are at a distance less than $\gve$ from its leaves. Then:
 \begin{align*}
\int_{G\setminus K_\gve} \eta_\gve(x)^p \dx &= \sum_{s\in \Gamma\setminus K^\circ} \int_{s-\gve}^s v_s^p\varphi_\gve^s(x)^p \dx = \sum_{s\in \Gamma\setminus K^\circ} v_s^p
\end{align*} 

In a similar way, it is easy to see that:
\[\lim_{\gve\to 0^+} \alpha_{c,K_\gve}= \alpha_K\]
Moreover, we have that f $K_1,K_2\in\mathcal{S}_a$ and $K_1\subset K_2$, then $\alpha_{K_1}\geq\alpha_{K_2}$. In particular, this means $\alpha_{c,K_\gve}\geq\alpha_{c,K}=\alpha_K$. This implies:
\[B_\gve \le B\quad \textrm{ and }\quad \lim_{\gve\to 0^+} B_\gve >B-\delta.\]
Since this can be done for every $\delta>0$, we have: $B=B_c$.
\qed\end{proof}


\bibliographystyle{unsrt}

\end{document}